\documentclass[11pt]{article}
\usepackage[latin9]{inputenc}
\usepackage[left=2cm,top=2cm,right=2cm,bottom=2cm]{geometry}
\usepackage{amsmath}
\usepackage{amsthm}
\usepackage{amssymb}
\usepackage{hyperref}
\usepackage{color}

\makeatletter
\numberwithin{equation}{section}
\usepackage{amsfonts}
\usepackage{graphicx}
\usepackage{amsthm}
\usepackage{authblk}

\theoremstyle{theorem}
\newtheorem{proposition}{Proposition}
\newtheorem{theorem}{Theorem} 
 
\newtheorem{lemma}{Lemma}
\theoremstyle{definition}
\newtheorem{assumption}{Assumption} 
\theoremstyle{remark}
\newtheorem{example}{Example} 
 
\newtheorem{remark}{Remark}

\newcommand\mikko[1]{\textcolor{black}{#1}}

\begin{document}
\title{Limit Theorems for Trawl Processes} 
\author[1]{Mikko \mikko{S.} Pakkanen\thanks{m.pakkanen@imperial.ac.uk}} 
\author[2]{Riccardo Passe\mikko{g}geri\thanks{riccardo.passeggeri@gmail.com}} 
\author[3]{Orimar Sauri\thanks{osauri@math.aau.dk}} 
\author[1]{Almut E. D. Veraart\thanks{a.veraart@imperial.ac.uk}} 
\affil[1]{Department of Mathematics, Imperial College London\protect\\ \mikko{South Kensington Campus}, London, SW7 2AZ, UK}
\affil[2]{LPSM, Sorbonne University, 4 Place Jussieu, Paris, 75005, France}
\affil[3]{Department of Mathematical Sciences, Aalborg University \protect\\ Skjernvej 4A, 9220, Aalborg, Denmark} 
\date{This Version: \today}
\maketitle
\begin{abstract}
In this work we derive limit theorems for trawl processes. First,
we study the asymptotic behaviour of the partial sums of the discretized
trawl process $(X_{i\Delta_{n}})_{i=0}^{\lfloor nt\rfloor-1}$, under the assumption that as $n\uparrow\infty$, $\Delta_{n}\downarrow0$ and $n\Delta_{n}\rightarrow\mu\in[0,+\infty]$. Second, we derive a functional limit theorem for trawl processes as the L\'{e}vy measure of the trawl seed grows to infinity and show that the limiting process has a Gaussian moving average representation.
\end{abstract}

\section{Introduction}

In this paper, we \mikko{study probabilistic limit theorems for a class of stationary infinitely divisible stochastic processes called \emph{trawl processes}, which were introduced for the first time in 2011 by Barndorff-Nielsen \cite{BN2011}.} \mikko{By construction, a trawl} process allows for both a very flexible autocorrelation structure and the possibility of generating any kind of marginal distribution within the class of  infinitely divisible distributions.  In particular, the marginal distribution and the autocorrelation structure can be modelled independently \mikko{of} each other. \mikko{Often the marginal distribution is chosen among infinitely divisible distributions on positive integers, with a view to applying the process as a model of serially correlated temporal count data, although in general such an assumption is not necessary.}

\mikko{Barndorff-Nielsen et al.\ \cite{BNLundShepVerr14}} provide the first systematic study of trawl processes, investigating their probabilistic properties and analysing volatility modulation within this framework.  Since this paper appeared, there has been an \mikko{increasing interest in} trawl processes\mikko{, covering a wide range of issues ranging from} applications \mikko{to} theoretical investigations\mikko{, and for the convenience of the reader} we provide here a brief \mikko{review of the recent literature on these} processes.

\mikko{Prior to the present paper, several limit theorems for trawl processes have been derived.} \mikko{Doukhan et al.\ \cite{DoukJakLopSurg18} characterize} a class of discrete time stationary trawl processes and study \mikko{the functional limits of} their partial sums. \mikko{Grahovac et al.\ \cite{GRAHOVAC2018235} investigate} the intermittency property of trawl process\mikko{, while Paulauskas \cite{Paulau}  investigates} trawl processes (and general linear processes) with tapered innovations. \mikko{Additionally, Talarczyk and Treszczotko \cite{TalTre2019} study} limit theorems for integrated trawl processes with symmetric L\'{e}vy bases.

\mikko{In a more applied realm,} Noven et al.\ \cite{Noven} \mikko{develop} a latent trawl process model for extreme values and appl\mikko{y} it to environmental time series. This work is partially extended \mikko{by Courgeau and Veraart} \cite{Courgeau}, \mikko{who derive} an asymptotic theory for \mikko{inference on} the latent trawl model for extreme values. \mikko{F}urther work in the direction of extreme values \mikko{has been done by Bacro et al.\ \cite{Bacro}, who propose} hierarchical space-time modelling of asymptotically independent exceedances based on a space-time extension of the trawl proces\mikko{s} and apply their model to precipitation data. \mikko{In finance, Shephard and Yang} \cite{SheYan} \mikko{and Veraart \cite{VERAART2019}} adapt the trawl process to provide a coherent \mikko{statistical} model \mikko{of} high-frequency data\mikko{, while t}he suitability of trawl processes \mikko{for the} modelling \mikko{of} high-frequency data is \mikko{further} corroborated by the results \mikko{of Rossi and Santucci de Magistris} \cite{RosSan}.
 
\mikko{With regards to estimation methodology} for trawl processes, in addition to the \mikko{afore}mentioned works \cite{Bacro,BNLundShepVerr14,Courgeau,Noven,SheYan}, \mikko{Doukhan \cite{Doukhan2020} introduces}  spectral estimation for non-linear long range dependent discrete time trawl processes, and \mikko{Shephard and Yang \cite{Shephard2016} develop} likelihood inference for exponential-trawl processes.

\mikko{In our paper} we \mikko{study two types of} limit theorems for trawl processes. \mikko{Our first main} result \mikko{concerns} the asymptotic behaviour of the partial sums of the discretized trawl process as the \mikko{size of the} discretization \mikko{step} goes to zero. In particular, let $L$ be an homogeneous L\'{e}vy basis on $\mathbb{R}^2$, let $a:\mathbb{R}^{+}\rightarrow\mathbb{R}^{+}$
be a non-increasing integrable function and let $A=\left\{ (r,y):r\leq0,0\leq y\leq a(-r)\right\}$. Then, $X_{t}:=L(A_{t}),\,\,\,t\in\mathbb{R}$, where $A_{t}:=A+(t,0)$, is termed the \textit{trawl process}. Let $(\Delta_n)_{n\in\mathbb{N}}$ be a sequence of non-negative constants such that $\Delta_{n}\downarrow0$ and $n\Delta_{n}\rightarrow\mu\in[0,+\infty]$ as $n\uparrow\infty$. We study the asymptotic behaviour of the (properly rescaled) functional
\[
\mathbf{S}^n = \Bigg(\sum_{k=0}^{\lfloor nt\rfloor-1}(X_{\Delta_n k}-\mathbb{E}(X_{\Delta_n k}))\Bigg)_{t\geq0}.
\]
The asymptotic behaviour of this functional depends on the value of $\mu$. Thus, we divide our analysis in three cases $0<\mu<\infty$, $\mu=0$, and $\mu=+\infty$. When $0<\mu<\infty$ we obtain that the above functional becomes a Riemann sum and thus we obtain a functional convergence in probability to $\int_{0}^{t\mu}X_{t}-\mathbb{E}(X_{t})ds$. 
\mikko{In the case $\mu=0$ it turns out that the behaviour of $\mathbf{S}^n$ depends on the increments of $X$ around 0. Based on this, we show that $\mathbf{S}^n$, after centering and properly rescaling, converges stably to certain stochastic integral driven by two independent L\'evy processes.}
%When $\mu=0$ we obtain a weak law of large numbers for fixed $t$ and, when studying the second order asymptotics, we obtain two stable finite dimensional distributions convergence depending on whether the Gaussian component of the trawl process is greater than or equal to zero. 

\mikko{Lastly, w}hen $\mu=+\infty$ the limit depends on whether the trawl process \mikko{$X$} has short or long memory. \mikko{Under short memory}, we \mikko{show} that, when properly scaled, \mikko{$\mathbf{S}^n$} converges to a Brownian motion. \mikko{In contrast, when $X$ exhibits long memory} we have to further distinguish whether the Gaussian component of the trawl process is \mikko{present or not}. If the Gaussian component is present, then \mikko{$\mathbf{S}^n$ under proper scaling converges towards} a fractional Brownian motion \mikko{w}ith Hurst parameter $H>1/2$. 
\mikko{Interestingly, if the Gaussian component is absent, the limit is no longer Gaussian and} the rate of convergence \mikko{of $\mathbf{S}^n$ is governed by} the Blumenthal-Getoor index of the trawl process. \mikko{We note that these findings agree with those obtained by Grahovac et al.\ \cite{GLST2019} on superpositions of Ornstein-Uhlenbeck type processes.}

\mikko{Our second main result} is a functional \mikko{limit theorem} for trawl processes \mikko{that links them with stationary Gaussian processes}. In particular, we show that the sequence of \mikko{scaled} trawl processes converges in distribution\mikko{,} as their L\'{e}vy measure\mikko{s} tend to infinity\mikko{,} \mikko{to} a limiting process which \mikko{admits} a Gaussian moving average representation. Moreover, we are able to show explicitly the relation between the function that determines the \mikko{upper boundary of the} trawl set, namely the function $a(\cdot)$ introduced \mikko{above}, and the kernel function of the Gaussian moving average. We \mikko{stress} that the existence of the Gaussian moving average representation of the limiting process is explicitly proved.

%Further, we provide various examples showing the applicability of this result. These examples include the case of sequences of trawl processes with Poisson distribution for which the intensity parameter goes to infinity.

%From the point of view of the limiting processes, this paper shows how the trawl process can be used as the building block of a wide range of processes which exhibit very different and desirable properties. From the point of view of trawl processes, this paper provides numerous crucial results on the their asymptotic behaviour, thus complementing and extending the existing theoretical results on trawl processes. Moreover, .

The paper is structured as follows. Section \ref{PreliminariesLBID} \mikko{lays out the notations used throughout the} paper \mikko{and discusses some essential} preliminaries. \mikko{In Sections \ref{Section-partial sums} and \ref{Sec-Moving average} we formulate the main results of the paper, concerning the asymptotics of partials sums of trawl processes and the convergence of a sequence of trawl processes to a Gaussian moving average, respectively.}
\mikko{For the sake of ease of exposition, we defer} the proofs of the\mikko{se} results to the end of the paper, namely to Section \ref{sec:Proofs}. Finally, the Appendix contains the computation of the fourth moment of the trawl process.
\section{Preliminaries\label{PreliminariesLBID}}

This part is devoted to introduce the basic notations as well as to
recall several basic results and concepts that will be used through
this paper.

\subsection{Functions of regular variation }

A function $g:\mathbb{R}^{+}\rightarrow\mathbb{R}^{+}$ is said to
be\textit{ regularly varying at $\infty$} with index $\alpha\in\mathbb{R}$
if as $t\rightarrow\infty$
\[
\frac{g(tx)}{g(t)}\rightarrow x^{-\alpha},\,\,\,\forall\,x>0.
\]
In this case we will write $g\in\mathrm{RV}_{\alpha}^{\infty}$. If
we replace $t\rightarrow\infty$ by $t\rightarrow0^{+}$ in the previous
equation, then $g$ is called \textit{regularly varying at $0$} and
in this case we denote this as $g\in\mathrm{RV}_{\alpha}^{0}$. In
the previous definitions $\alpha=0$, typically we will refer to $g$
as \textit{slowly varying}. It is well known that if $g\in\mathrm{RV}_{0}^{\infty}$,
then as $x\rightarrow\infty$
\[
g(x)x^{\varepsilon}\rightarrow\begin{cases}
+\infty & \text{if }\varepsilon>0;\\
0 & \text{if }\varepsilon<0.
\end{cases}
\]
One of the key results for functions of regular variation is the so
\textit{Karamata's Theorem} (KT for short) which states that if $g\in\mathrm{RV}_{\alpha}^{\infty}$
and locally bounded in $[x_{0},+\infty)$ then 
\begin{enumerate}
\item If $\rho\geq\alpha-1$, it holds that 
\[
\frac{1}{x^{\rho+1}g(x)}\int_{x_{0}}^{x}g(s)s^{\rho}\mathrm{d}s\rightarrow\frac{1}{\rho-\alpha+1},\,\,\,x\rightarrow\infty.
\]
\item For every $\rho<\alpha-1$, we have that 
\[
\frac{1}{x^{\rho+1}g(x)}\int_{x}^{\infty}g(s)s^{\rho}\mathrm{d}s\rightarrow\frac{1}{\alpha-1-\rho},\,\,\,x\rightarrow\infty.
\]
\end{enumerate}
For a complete exposition on the basic properties of functions of
regular variation we refer the reader to \cite{BinghamGoldTeu87}.

\subsection{Stable convergence}

For the rest of this paper $\left(\Omega,\mathcal{F},\mathbb{P}\right)$
will denote a complete probability space. The notations $\overset{\mathbb{P}}{\rightarrow}$
and $\overset{d}{\rightarrow}$ stand, respectively, for convergence
in probability and distribution of random vectors (r.v.'s for short).
If $X_{n}/Y_{n}$$\overset{\mathbb{P}}{\rightarrow}$0 when $n\rightarrow\infty$,
we write $X_{n}=\mathrm{o}_{\mathbb{P}}(Y_{n})$ . Given a sub-$\sigma$-field
$\mathcal{G}\subseteq\mathcal{F}$ and a sequence of r.v.'s $(\xi_{n})_{n\geq1}$
on $\left(\Omega,\mathcal{F},\mathbb{P}\right)$, by\textit{ $\mathcal{G}$-stably
convergence in distribution} of $\xi_{n}$ towards a random vector
(r.v. for short) $\xi$ (in symbols $\xi_{n}\overset{\mathcal{G}\text{-}d}{\longrightarrow}\xi$),
we mean that, conditioned on any non-null event in $\mathcal{G}$,
$\xi_{n}$$\overset{d}{\rightarrow}\xi$. In this framework, if $(X_{t}^{n})_{t\in\mathbb{R},n\in\mathbb{N}}$
is a family of stochastic processes, we will write $X^{n}\overset{\mathcal{G}\text{-}fd}{\longrightarrow}X$
if the finite-dimensional distributions (f.d.d. for short) of $X^{n}$
converge $\mathcal{G}$-stably toward the f.d.d. of $X$. We further
write $X^{n}\overset{\mathcal{G}\text{-}\mathcal{D}[0,T]}{\Longrightarrow}X$,
if $X^{n}$ converges to $X$ in the Skhorohod topology and $X^{n}\overset{\mathcal{G}\text{-}fd}{\longrightarrow}X$.
We refer the reader to \cite{HauslerLuschgy15} for a concise exposition
of stable convergence.

\subsection{L\'{e}vy bases and infinite divisibility}

Let $\mu$ be a measure on $\mathcal{B}(\mathbb{R}^{d})$, the Borel
sets on $\mathbb{R}^{d}$, and let $\mathcal{B}_{b}^{\mu}(\mathbb{R}^{d}):=\{A\in\mathcal{B}(\mathbb{R}^{d}):\mu(A)<\infty\}.$
The family $L=\{L\left(A\right):A\in\mathcal{B}_{b}^{\mu}(\mathbb{R}^{d})\}$
of real-valued r.v.'s will be called a \textit{L\'{e}vy basis} if it is
an infinitely divisible (ID for short) independently scattered random
measure, that is, $L$ is $\sigma$-additive almost surely and such
that for any $A,B\in\mathcal{B}_{b}^{\mu}(\mathbb{R}^{d})$, $L(A)$
and $L(B)$ are ID r.v.'s that are independent whenever $A\cap B=\emptyset$.
The cumulant of a r.v. $\xi$, in case it exists, will be denoted
by $\mathcal{C}(z\ddagger\xi):=\log\mathbb{E}(e^{iu\xi})$. We will
say that $L$ is \textit{separable} with \textit{control measure}
$\mu$, if 
\[
\mathcal{C}(z\ddagger L(A))=\mu(A)\psi(z),\,\,\,A\in\mathcal{B}_{b}^{\mu}(\mathbb{R}^{d}),z\in\mathbb{R},
\]
where 
\[
\psi(z):=i\gamma z-\frac{1}{2}b^{2}z^{2}+\int_{\mathbb{R}\backslash\{0\}}(e^{izx}-1-izx\mathbf{1}_{\left|x\right|\leq1})\nu(\mathrm{d}x),\,\,\,z\in\mathbb{R},
\]
with $\gamma\in\mathbb{R},$ $b\geq0$ and $\nu$ is a L\'{e}vy measure,
i.e. $\nu(\{0\})=0$ and $\int_{\mathbb{R}\backslash\{0\}}(1\land\left|x\right|^{2})\nu(\mathrm{d}x)<\infty$.
When $\mu=Leb$, in which $Leb$ represents the Lebesgue measure on
$\mathbb{R}^{d}$, $L$ is called \textit{homogeneous}. The ID r.v.
associated to the characteristic triplet $\left(\gamma,b,\nu\right)$
is called the \textit{L\'{e}vy seed} of $L$ and will be denoted by $L'$.
As usual, $\left(\gamma,b,\nu\right)$ will be called the characteristic
triplet of $L$ and $\psi$ its characteristic exponent. The \textit{Blumenthal-Getoor
index} of an ID distribution with triplet $\left(\gamma,b,\nu\right)$,
is defined and denoted as 
\[
\beta_{\nu}:=\inf\left\{ \beta>0:\int_{\left|x\right|\leq1}\left|x\right|^{\beta}\nu(\mathrm{d}x)<\infty\right\} .
\]
Within this framework, we will also refer to $\beta_{\nu}$ as the
Bluementhal-Getoor index of a homogeneous L\'{e}vy basis with characteristic
triplet $\left(\gamma,b,\nu\right)$. In this paper, the sigma field
generated by $L$ is denoted by $\mathcal{F}_{L}$.

For any L\'{e}vy measure $\nu$, we associate the functions $\nu^{\pm}:(0,\infty)\rightarrow\mathbb{R}^{\text{+}}$,
defined as $\nu^{+}(x):=\nu(x,\infty)$ and $\nu^{-}(x):=\nu(-\infty,-x).$
Let $K_{+}+K_{-}>0$ and $0<\beta<2$. A separable L\'{e}vy basis is called
strictly $\beta$-stable with parameters $(\beta,K_{+},K_{-},\gamma)$
if its L\'{e}vy seed is distributed according to a strictly $\beta$-stable
distribution, that is, the characteristic triplet of $L'$ has no
Gaussian component ($b=0$), its L\'{e}vy measure satisfies 
\[
\frac{\nu(\mathrm{d}x)}{\mathrm{d}x}=K_{+}\left|x\right|^{-1-\beta}\mathbf{1}_{\{x>0\}}+K_{-}\left|x\right|^{-1-\beta}\mathbf{1}_{\{x<0\}},
\]
and either $\gamma=(K_{+}-K_{-})/\left|\beta-1\right|$ when $\beta\neq1$,
or $\gamma$ arbitrary with $K_{+}=K_{-}$ in the case of $\beta=1$.
The characteristic exponent of a strictly $\beta$-stable with parameters
$(\beta,K_{+},K_{-},\gamma)$ admits the representation for every
$z\in\mathbb{R}$
\begin{equation}
\psi(z;\beta,K_{+},K_{-},\gamma):=\begin{cases}
-\sigma\left|z\right|^{\beta}(1-\mathbf{i}\rho\mathrm{sign}(z)\tan(\pi\beta/2)) & \text{if }\beta\neq1;\\
-K_{+}\pi\left|z\right|+\mathbf{i}\gamma z & \text{if }\beta=1,
\end{cases}\label{chfnctstrictlystabe}
\end{equation}
where 
\[
\sigma:=\frac{\Gamma(2-\beta)}{\beta(1-\beta)}\cos(\pi\beta/2)(K_{+}+K_{-}),\,\,\,\text{ and }\,\,\,\rho:=\frac{K_{-}-K_{+}}{K_{+}+K_{-}}.
\]
\[
\]

\subsection{Trawl processes}

Let $L$ be a homogeneous L\'{e}vy basis on $\mathbb{R}^{2}$ with characteristic
triplet $(\beta,b,\nu)$. In addition, let $a:\mathbb{R}^{+}\rightarrow\mathbb{R}^{+}$
be a non-increasing integrable function and put 
\[
A=\left\{ (r,y):r\leq0,0\leq y\leq a(-r)\right\} .
\]
The process defined by
\begin{equation}
X_{t}:=L(A_{t}),\,\,\,t\in\mathbb{R},\label{trawldef}
\end{equation}
where $A_{t}:=A+(t,0)$, is termed as a \textit{trawl process}. From
now on, we will refer to $A$ and $a$, as the trawl set and the trawl
function, respectively. It is well known that $X$ is strictly stationary
and, in the case when $L$ is square integrable, its auto-covariance
function is given by
\begin{equation}
\varGamma_{X}(h):=Var(L')\int_{\left|h\right|}^{\infty}a(u)\mathrm{d}u,\,\,\,h\in\mathbb{R}.\label{TrawlACF}
\end{equation}
Moreover, $\varGamma_{X}$ uniquely characterizes $a$. More precisely,
if $L$ is square integrable, and $X$ and $\tilde{X}$ are two trawls
processes associated to $L$ with trawls functions $a$ and $\tilde{a}$,
respectively, then $a=\tilde{a}$ a.e. if and only if
\[
\varGamma_{X}=\varGamma_{\tilde{X}}.
\]
For a detailed exposition on the basic properties of trawl processes
we refer to \cite{BNLundShepVerr14} and \cite{BNBEnthVeraat18}.

\section{Limit theorems for partial sums of trawl processes}\label{Section-partial sums}

In this section we focus on the limit theorems for the partial sums
of $(X_{i\Delta_{n}})_{i=0}^{n-1}$ under the assumption that as $n\uparrow\infty$,
$\Delta_{n}\downarrow0$ and $n\Delta_{n}\rightarrow\mu\in[0,+\infty].$
More specifically, we study the asymptotic behavior of the process
$\mathbf{S}^{n}=(S_{\left[nt\right]}^{\Delta_{n}})_{t\geq0}$ , where
\[
S_{m}^{\Delta}:=\sum_{k=0}^{m-1}(X_{\Delta k}-\mathbb{E}(X_{\Delta k})),\,\,\,m\in\mathbb{N},\Delta>0,
\]
with $X$ as in (\ref{trawldef}). Note that we will always assume
that the associated L\'{e}vy basis $L$ has characteristic triplet $(\gamma,b,\nu)$
and $\mathbb{E}(\left|L^{\prime}\right|)<\infty$ and that $a$ is
continuous in $[0,\infty)$. Furthermore, for the sake of exposition
all \mikko{of} our proof\mikko{s} are presented in Section \ref{sec:Proofs}.

\subsection{Main results}

Through this part we state our main results \mikko{concerning} $\mathbf{S}^{n}.$
As expected, the rate of convergence will depend entirely on the sampl\mikko{ing}
scheme, which is in turn represented by $\mu$. In what follows we
will use the notation 
\[
\tilde{X_{t}}:=X_{t}-\mathbb{E}(X_{t}),\,\,\,t\in\mathbb{R}.
\]

\subsubsection{$0<\mu<\infty$}

Let us start assuming that $n\Delta_{n}\rightarrow\mu\in(0,\infty)$.
In this situation the points
\[
t_{i}=i\Delta_{n},\,\,\,i=0,\ldots,\left[nt\right]-1,
\]
form a partition of $[0,t\mu]$. Consequently, $\Delta_{n}\mathbf{S}_{t}^{n}$
becomes a Riemann sum for the mapping $s\mapsto\tilde{X}_{s}$. Based
on this observation, the following result is not surprising.

\begin{proposition}\label{propmufinite}

Suppose that $\mathbb{E}(\left|L^{\prime}\right|^{2})<\infty$ and
$\Delta_{n}n\rightarrow\mu\in(0,\infty)$. Then for every $V>0$
\[
\sup_{0\leq t\leq V}\left|\Delta_{n}\mathbf{S}_{t}^{n}-\int_{0}^{t\mu}\tilde{X_{s}}ds\right|\overset{\mathbb{P}}{\rightarrow}0.
\]

\end{proposition}

\subsubsection{$\mu=0$}

Let us now turn our attention to the case when $\mu=0$. Intuitively,
when this occur one should expect that
\[
X_{\Delta_{n}n}\approx X_{\Delta_{n}i}\approx X_{0},\,\,\,i=0,1,\ldots,n-1,
\]
for $n$ large, which suggests that 
\[
\frac{1}{n}S_{n}^{\Delta_{n}}\approx\tilde{X}_{0}.
\]
This turn out to be true as the following result shows.

\begin{proposition}\label{porpndeltnzero}

Suppose that $\mathbb{E}(\left|L^{\prime}\right|)<\infty$ and $a$
is continuously differentiable on a neighborhood of $0$. If $\Delta_{n}n\rightarrow0$
as $n\rightarrow\infty$, then 
\[
\frac{1}{n}\mathbf{S}_{t}^{n}\overset{\mathbb{P}}{\rightarrow}t\tilde{X}_{0},\,\,\,t\geq0.
\]
\end{proposition}

Next, we proceed to derive second order asymptotics for $\mathbf{S}^{n}$
when $\mu=0$. Following the previously discussed heuristic, one should
expect that for large $n$
\[
\frac{1}{n}\mathbf{S}_{t}^{n}-t\tilde{X}_{0}\approx t(X_{\Delta_{n}n}-X_{0}).
\]
Therefore, in this case the asymptotics are determined by the behavior
of the increments of $X$. Before presenting our results in this framework,
we introduce our working assumption, which reads as 

\begin{assumption}\label{assumptionmu0}There is $0<\beta<2$ such
that (see Section \ref{PreliminariesLBID}) $\nu^{\pm}(x)\sim\tilde{K}_{\pm}x^{-\beta}$
as $x\rightarrow0^{+}$ with $\tilde{K}_{+}+\tilde{K}_{-}>0$. Furthermore,
if $\beta=1$ assume in addition that $\tilde{K}_{+}=\tilde{K}_{-}$
and $\mathrm{PV}\int_{-1}^{1}x\nu(\mathrm{d}x)$, the Cauchy principal
value, exists.

\end{assumption}

\begin{theorem}\label{thmndeltnzero}Let the assumptions of Proposition
\ref{porpndeltnzero} hold and put $Z_{t}^{n}:=\left(\frac{\left[nt\right]}{n}X_{0}-\frac{1}{n}\mathbf{S}_{t}^{n}\right)$
Then the following holds 
\begin{description}
\item [{i.}] If $b>0$
\[
\frac{1}{\sqrt{n\Delta_{n}}}Z_{t}^{n}\overset{\mathcal{F}\text{-}fd}{\longrightarrow}\sigma\int_{0}^{t}(t-r)\mathrm{d}(B_{r}^{(1)}+B_{r}^{(2)}),\,\,\,t\geq0,
\]
where $B^{(1)}$, $B^{(2)}$ are two independent Brownian motions
which are in turn independent of $L$, and $\sigma=b^{2}a(0)$.
\item [{iii.}] Suppose that $b=0$ and Assumption \ref{assumptionmu0}
holds. Then as $n\rightarrow\infty$
\[
\frac{1}{(n\Delta_{n})^{1/\beta}}Z_{t}^{n}\overset{\mathcal{F}\text{-}fd}{\longrightarrow}\sigma\int_{0}^{t}(t-s)\mathrm{d}(Y_{s}^{(1)}-Y_{t}^{(2)}),\,\,\,t\geq0,
\]
where and $Y^{(1)}$ and $Y^{(2)}$ two i.i.d. strictly $\beta$-stable
L\'{e}vy processes independent of $L$ satisfying that 
\[
\mathcal{C}(z\ddagger Y_{1}^{(q)})=\psi(z;\beta,\varrho\tilde{K}_{+},\varrho\tilde{K_{-}},\tilde{\gamma}),\,\,\,z\in\mathbb{R},
\]
with $\varrho=\beta a(0)$ and $\tilde{\gamma}=\gamma+\mathrm{PV}\int_{-1}^{1}x\nu(\mathrm{d}x)$
when $\beta=1$.
\end{description}
\end{theorem}

\subsubsection{The case when $\mu=+\infty$}

Suppose now that $n\Delta_{n}\rightarrow\mu=+\infty$ as $n\uparrow\infty$
and $\Delta_{n}\downarrow0$. In order to get some intuition of what
one should expect in this situation, firstly let $\Delta_{n}=\Delta$,
i.e. the space between observations if fixed. Obviously, $\Delta_{n}n\rightarrow+\infty$
and the process $(X_{\Delta n})_{n\geq1}$ is strictly stationary.
In this situation $\mathbf{S}^{n}$ becomes the partial sums of a
discrete-time stationary process. In view of this, when properly scaled,
$\mathbf{S}^{n}$ typically converges to either a Brownian motion
or a fractional Brownian motion (fBm for short), depending whether
$(X_{\Delta n})_{n\geq1}$ has short memory or long memory, respectively.
It turned out that in our general setup, the former case the same
result holds, while in the latter $\mathbf{S}^{n}$ will converge
to a fBm only when $L$ has a Gaussian component. Before presenting
our results for this sampling scheme, we introduce our working assumptions.

\begin{assumption}[SM]\label{assumptionshortmem}

There is $p_{0}>2$ such that $\mathbb{E}(\left|L^{\prime}\right|^{p_{0}})<\infty$
and $a(s)=\mathrm{O}(s^{-p_{0}})$ as $s\uparrow+\infty$.

\end{assumption}

\begin{assumption}[LM]\label{assumptionlongmem}

Assume that $\mathbb{E}(\left|L^{\prime}\right|)<\infty$, and that
there is a strictly positive continuous function $a^{\prime}\in\mathrm{RV}_{\kappa}^{\infty}$,
with $2<\kappa<3$, such that 
\[
a(s)=\int_{s}^{\infty}a^{\prime}(y)dy,\,\,\,s\geq0.
\]

\end{assumption}

\begin{assumption}[LM']\label{assumptionlongmem-1}

Assumption \ref{assumptionlongmem} holds and for some $c_{a}>0$,
$a^{\prime}(y)\sim c_{a}y^{-\kappa}$ as $y\rightarrow\infty$.

\end{assumption}

Our first result concerns to the short memory case:

\begin{theorem}\label{theoremmean}

Suppose that $\mu=+\infty$, and that Assumption \ref{assumptionshortmem}
is fulfilled. Put 
\[
\mathcal{G}^{X}=\sigma(\cup_{k\geq1}\cap_{N\geq k}\sigma(X_{0},X_{\Delta_{N}},\ldots,X_{k\Delta_{N}})).
\]
Then, as $n\uparrow\infty$
\[
\sqrt{\frac{\Delta_{n}}{n}}\mathbf{S}^{n}\overset{\mathcal{\mathcal{G}}^{X}-\mathcal{D}[0,1]}{\Longrightarrow}\sigma_{a}B,
\]
where $\sigma_{a}^{2}=Var(L^{\prime})\int_{\mathbb{R}}a(s)\mathrm{d}s$
and $B$ is a Brownian motion independent of $\mathcal{\mathcal{G}}^{X}$.

\end{theorem}

\begin{remark}By the independent scattered property of $L$, the
limiting process appearing in Theorem \ref{theoremmean} is not only
independent of $\mathcal{\mathcal{G}}^{X}$, but also of 
\[
\sigma(L(B):B\cap\cup_{t\geq0}A_{t}=\emptyset).
\]
Furthermore, in view that the array of $\sigma$-fields
\[
\mathcal{F}_{j,n}^{X}:=\sigma(X_{0},X_{\Delta_{n}},\ldots,X_{j\Delta_{n}}),\,\,\,\,j=0,1,\ldots,n-1,
\]
is ``almost nested'', we conjecture that 
\[
\mathcal{\mathcal{G}}^{X}=\sigma(X_{t},t\geq0).
\]

\end{remark}

The asymptotic behavior drastically changes when $X$ has long memory.
In order to have better exposition of our results, we split our finding
in two theorems that distinguish the case in whether the Gaussian
component in $L$ is present or not. 

\begin{theorem}\label{theoremmean-2}

Let Assumption \ref{assumptionlongmem} hold. Suppose that $b>0$,
$\mathbb{E}(\left|L^{\prime}\right|^{2})<\infty$, and $\mu=+\infty$.
Then as $n\uparrow\infty$ 
\[
\frac{1}{n\sqrt{a(n\Delta_{n})n\Delta_{n}}}\mathbf{S}^{n}\overset{\mathcal{\mathcal{G}}^{X}-\mathcal{D}[0,1]}{\Longrightarrow}\sigma_{\kappa}B^{H},\,\,\,n\rightarrow\infty.
\]
where $\sigma_{\kappa}^{2}=\frac{b^{2}}{(\kappa-2)(3-\kappa)(4-\kappa)}$,
$B^{H}$ is a fBm of index $H=2-\frac{\kappa}{2}>1/2$ and $\mathcal{\mathcal{G}}^{X}$
as in Theorem \ref{theoremmean}.

\end{theorem}

When the Gaussian component is not present, the limit is not anymore
Gaussian and the rate of convergence for $\mathbf{S}^{n}$ varies
according to the behavior of $\beta_{\nu}$ the Blumenthal-Getoor
index of $L$. More precisely:

\begin{theorem}\label{theoremmean3}

Let Assumption \ref{assumptionlongmem-1} hold. Suppose that $b=0$,
$\mathbb{E}\left(\left|L^{\prime}\right|^{2}\right)<\infty$, and
that $\mu=+\infty$. The following holds:
\begin{description}
\item [{i.}] If $\beta_{\nu}<\kappa-1$, then, as $n\uparrow\infty$
\[
\frac{\Delta_{n}}{\left(c_{a}n\Delta_{n}\right)^{\frac{1}{\kappa-1}}}\mathbf{S}^{n}\overset{fd}{\rightarrow}Y,
\]
where $Y$ is a strictly $(\kappa-1)$-stable L\'{e}vy process satisfying
that (see (\ref{chfnctstrictlystabe}) 
\[
\mathcal{C}(z\ddagger Y_{1})=\psi(z;\kappa-1,K_{+,\kappa},K_{+,\kappa},\tilde{\gamma}),
\]
with $K_{+,\kappa}=\int_{0}^{\infty}x^{\kappa-1}\nu(\mathrm{d}x)$
and $K_{-,\kappa}=\int_{-\infty}^{0}\left|x\right|^{\kappa-1}\nu(\mathrm{d}x)$.
\item [{ii.}] When $2>\beta_{\nu}>\kappa-1>1$, further assume that for
some $\tilde{K}_{+}+\tilde{K}_{-}>0$, $\nu^{\pm}(x)\sim\tilde{K}_{\pm}x^{-\beta_{\nu}}$
as $x\rightarrow0^{+}$. Then, as $n\rightarrow\infty$
\[
\frac{1}{n\left(a(n\Delta_{n})n\Delta_{n}\right)^{1/\beta_{\nu}}}\mathbf{S}_{1}^{n}\overset{d}{\rightarrow}\xi,
\]
where $\xi$ is a strictly $\beta_{\nu}$-stable such that 
\[
\mathcal{C}(z\ddagger Y_{1})=\psi(z;\beta_{\nu},K_{+,\kappa,\beta_{\nu}},K_{-,\kappa,\beta_{\nu}},\tilde{\gamma}),
\]
in which $K_{\pm,\kappa,\beta_{\nu}}=\varrho_{a}\beta_{\nu}\tilde{K}_{\pm}$
and 
\[
\varrho_{a}=\frac{1}{\kappa-2}+(\kappa-1)\int_{0}^{1}(1-s)s^{\beta_{\nu}-\kappa}\mathrm{d}s+2\int_{0}^{1}s^{\beta_{\nu}-\kappa+1}\mathrm{d}s.
\]
\end{description}
\end{theorem}

\begin{remark}The notation $\nu^{\pm}(x)\sim\tilde{K}_{\pm}x^{-\beta}$
here means that $x^{\beta}\nu^{\pm}(x)\rightarrow\tilde{K}_{\pm}$
when $x\downarrow0$. Moreover, such property only concerns to the
behavior of the L\'{e}vy measure of $L$ around zero. Hence, one can have
simultaneously that this condition is satisfied and that the second
moment of $L$ is finite. An example of such infinitely divisible
distribution is the normal inverse Gaussian distribution (see \cite{Sauri17}).

\end{remark}

Most of our estimates used in the proof of the previous theorem heavily
rely on the square integrability of $L$. Thus, it is natural to consider
the situation in which this condition does not hold anymore. The following
result give a partial answer to this matter.

\begin{theorem}\label{theoremmean-2-1}

Let Assumption \ref{assumptionlongmem-1} hold. Suppose that $L$
is strictly $\beta$-stable with parameters $(K_{+},K_{-},\beta,\hat{\gamma})$
and that $\mu=+\infty$. Then the following holds:
\begin{description}
\item [{i.}] If $1<\beta<\kappa-1$ , then 
\[
\frac{\Delta_{n}}{\left(n\Delta_{n}\right)^{1/\beta}}\mathbf{S}^{n}\overset{fd}{\rightarrow}Y,\,\,\,n\rightarrow\infty.
\]
where $Y$ is a strictly $\beta$-stable L\'{e}vy process satisfying that
(see (\ref{chfnctstrictlystabe}))
\[
\mathcal{C}(z\ddagger Y_{1})=\mathcal{C}(z\ddagger Y_{1})=\psi(z;\beta,\varrho_{a}K_{+,a},\varrho_{a}K,\hat{\gamma}),
\]
where $\varrho_{a}=\int_{0}^{\infty}s^{\beta}a^{\prime}(s)\mathrm{d}s$.
\item [{ii.}] If $2>\beta>\kappa-1$ then the conclusion in Theorem \ref{theoremmean3}
ii. remains valid.
\end{description}
\end{theorem}

\section{Convergence to a Gaussian moving average}\label{Sec-Moving average}
In this section we show that under certain assumptions a sequence of trawl processes converge to a Gaussian moving average. In particular, the main theorem of this section explores the case where the L\'{e}vy measure of the L\'{e}vy seed of the trawl process explodes as $n\rightarrow\infty$.

Let $T>0$, let $r,s,t\in[0,T]$ with $r\leq s\leq t$, and let $\tilde{B}_{t,s,r}:=A_{s}\setminus A_{t}\setminus A_{r}$ (see also Figure 1), namely
\begin{equation*}
Leb(\tilde{B}_{t,s,r})=\int_{r}^{s}a(s-p)-a(t-p)dp.
\end{equation*} Consider the following general assumption on the behaviour of the trawl set.
\begin{assumption}[On the behaviour of trawl sets]\label{A1}
	We assume that $a$ is monotone, that given any $t,s\in\mathbb{R}$, $Leb(A_{t}\setminus A_{s})\leq C(t-s)^{\frac{1}{2}+\frac{\epsilon}{2}}$ and that $Leb(\tilde{B}_{t,s,r})\leq C_{T}(t-r)^{1+\epsilon}$, where $C,\in(0,\infty)$, $C_{T}\in(0,T^{-(1+\epsilon)}]$ and $\epsilon$ can take any value in $(0,\infty)$.
\end{assumption}
\noindent Observe that under the above assumption we have $Leb(\tilde{B}_{t,s,r})\leq 1$, hence $Leb(\tilde{B}_{t,s,r})\geq (Leb(\tilde{B}_{t,s,r}))^{\beta}$ for any $\beta\geq 1$. 
\begin{remark}
	This assumption is only needed to prove tightness in the proof of Theorem \ref{TheSpecific}, since finite dimensional distribution convergence does not rely on it.
\end{remark}
\begin{example}[Exponential]
	For $p\geq0$ consider $a(p):=C_{T} e^{-p}$ with $C_{T}\in(0, T^{-2}]$, then by the mean value theorem we have that
	\begin{equation*}
	Leb(\tilde{B}_{t,s,r})=C_{T} \int_{r}^{s}e^{p-s}-e^{p-t}dp\leq C_{T} \int_{r}^{s}(t-s)e^{p-s}dp\leq C_{T} (s-r)(t-s)\leq C_{T} (t-r)^{2}\leq 1,
	\end{equation*}
	and
	\begin{equation*}
	Leb(A_{t}\setminus A_{s})=C_{T} \int_{s}^{t}e^{p-t}dp\leq C_{T}(t-s).
	\end{equation*}
\end{example}
\begin{example}[Bounded first derivative on compact intervals]\label{Example}
	Consider a monotone function $f:\mathbb{R}_{+}\to\mathbb{R}_{+}$ with $\sup\limits_{p\in[0,T]}|f'(p)|\leq M$ where $M\in(0,\infty)$. Consider $a(p):=C_{T} f(p)$ with $C_{T}\in(0, \frac{1}{MT^{2}}]$, then by the mean value theorem we have that
	\begin{equation*}
	Leb(\tilde{B}_{t,s,r})=C_{T} \int_{r}^{s}f(s-p)-f(t-p)dp\leq C_{T} \int_{r}^{s}(t-s)\max\limits_{x\in[s-p,t-p]} |f'(x)|dp
	\end{equation*}
	\begin{equation*}
	\leq C_{T}M \int_{r}^{s}(t-s)dp\leq C_{T}M (s-r)(t-s)\leq C_{T}M (t-r)^{2}\leq 1,
	\end{equation*}
	and
	\begin{equation*}
	Leb(A_{t}\setminus A_{s})=C_{T} \int_{s}^{t}f(t-p)dp\leq C(t-s).
	\end{equation*}
\end{example}
\begin{example}[$a(p)=(p)^{-\frac{1}{2}+\frac{\epsilon}{2}}$ case]
	Let $a(p)=(p)^{-\frac{1}{2}+\frac{\epsilon}{2}}$ (for $p$ small) then we have
	\begin{equation*}
	Leb(\tilde{B}_{t,s,r})= \int_{r}^{s}(s-p)^{-\frac{1}{2}+\frac{\epsilon}{2}}-(t-p)^{-\frac{1}{2}+\frac{\epsilon}{2}}dp
	\end{equation*}
	\begin{equation*}
	=C\left[(s-r)^{\frac{1}{2}+\frac{\epsilon}{2}}-(t-r)^{\frac{1}{2}+\frac{\epsilon}{2}}+(t-s)^{\frac{1}{2}+\frac{\epsilon}{2}}\right],
	\end{equation*}
	and notice that when $s=(t+r)/2$ (namely $t-s=s-r$) by denoting $x:=t-s$ then 
	\begin{equation*}
	x^{\frac{1}{2}+\frac{\epsilon}{2}}-(2x)^{\frac{1}{2}+\frac{\epsilon}{2}}+x^{\frac{1}{2}+\frac{\epsilon}{2}}=\left(2-2^{\frac{1}{2}+\frac{\epsilon}{2}}\right)x^{\frac{1}{2}+\frac{\epsilon}{2}}=C' x^{\frac{1}{2}+\frac{\epsilon}{2}},
	\end{equation*}
	which does \textbf{not} satisfies the desired condition of Assumption \ref{A1}.
\end{example}
In the above examples we have given a particular structure to the trawl function $a$ and then check whether such $a$ satisfies Assumption \ref{A1} or not. However, there is another modelling point of view we can take. Imagine that we would like to approximate a moving average with a particular kernel by a sequence of trawl processes. How can we choose the right $a$? In other words, how do we choose the right sequence of trawl processes? In the next result, which is the main result of this section, we answer this question too, namely we provide a link between $a$ and the kernel function of the moving average.
\begin{theorem}\label{TheSpecific}
	Let $a(h)=-\frac{d}{dh}\int_{0}^{\infty}g(s)g(h+s)ds$ for every $h\geq0$ where $g$ is an integrable function. Let $X_{t}^{(n)}$ be the associated trawl process with characteristics $(\gamma^{(n)},b^{(n)},\nu^{(n)})$, $n\in\mathbb{N}$. Assume that Assumption \ref{A1} holds and that $\int_{\mathbb{R}}\frac{x^{2}}{n}\nu^{(n)}(dx)\rightarrow 1$, $\int_{\mathbb{R}}\frac{|x^{3}|}{n\sqrt{n}}\nu^{(n)}(dx)\rightarrow 0$, $\int_{\mathbb{R}}\frac{x^{4}}{n^{2}}\nu^{(n)}(dx)$ is uniformly bounded, and $\frac{(b^{(n)})^{2}}{n}\rightarrow0$ as $n\rightarrow\infty$. Then we have
	\begin{equation*}
	\left\{\frac{1}{\sqrt{n}}\left(X_{t}^{(n)}-\mathbb{E}[X_{t}^{(n)}]\right)\right\}_{t\in[0,T]}\stackrel{d}{\rightarrow}\left\{\int_{-\infty}^{t}g(t-s)dB_{s}\right\}_{t\in[0,T]},
	\end{equation*}
	as $n\rightarrow\infty$, where $B$ is a one dimensional Brownian motion.
\end{theorem}
\begin{remark}
	Observe that in case we have $g$ is differentiable and positive monotone (bounded) then by the monotone (bounded) convergence theorem we have $\frac{d}{dh}\int_{0}^{\infty}g(s)g(h+s)ds=\int_{0}^{\infty}g(s)g'(h+s)ds.$
\end{remark}
\begin{remark}
	The theorem can be generalised using a general sequence of real numbers $(a_{n})_{n\in\mathbb{N}}$ instead of $(\frac{1}{\sqrt{n}})_{n\in\mathbb{N}}$. However, since the computations are exactly the same we decided to leave it with the more natural notation $\frac{1}{\sqrt{n}}$.
\end{remark}
\begin{example}[The Poisson case]\label{Ex}
	In this example we are going to show that the assumptions of Theorem \ref{TheSpecific} are satisfied for the case $X_{t}^{(n)}=L^{(n)}(A_{t})\sim \text{Poisson}(\lambda^{(n)} Leb(A))$ for all $t\in[0,T]$, where $\lambda^{(n)}$ is the intensity parameter, or equivalently, for the case $L'^{(n)}\sim \text{Poisson}(\lambda^{(n)})$. In particular, we have that
	\begin{equation*}
	\mathcal{C}(z\ddagger L'^{(n)})=\lambda^{(n)}\left(e^{iz}-1\right).
	\end{equation*}
	In order to satisfies the assumptions we have to impose that $\lambda^{(n)}=n+o(n)$ (\textit{e.g.} $\lambda^{(n)}=n+bn^{\gamma}$ for $b\in\mathbb{R}$ and $\gamma<1$). Indeed, 
	\begin{equation*}
	\int_{\mathbb{R}}\frac{x^{2}}{n}\nu^{(n)}(dx)\rightarrow 1\Leftrightarrow \frac{\lambda^{(n)}}{n}\rightarrow 1,\quad\textnormal{and}\quad \int_{\mathbb{R}}\frac{|x^{3}|}{n\sqrt{n}}\nu^{(n)}(dx)\rightarrow 0\Leftrightarrow \frac{\lambda^{(n)}}{n\sqrt{n}}\rightarrow 0,
	\end{equation*}
	as $n\rightarrow\infty$.
\end{example}
\begin{example}
	Let $g:\mathbb{R}_{+}\to\mathbb{R}_{+}$ be integrable, monotonically decreasing and second order differentiable with $g''(x)>0$, $\forall x\in\mathbb{R}_{+}$. Let $C_{T}=\left(T^{2}g''(0)\int_{0}^{\infty}g(s)ds\right)^{-1}$. Then $a(h)=-C_{T}\int_{0}^{\infty}g(s)g'(h+s)ds$ satisfies the assumptions of Theorem \ref{TheSpecific}. Indeed, it is possible to see that $a$ is positive (since $g'$ is negative), is monotonically decreasing (since $g'$ is monotonically increasing), and satisfies Assumption 5 thanks to Example \ref{Example}, indeed 
	\begin{equation*}
	\sup\limits_{p\in[0,T]}\frac{1}{C_{T}}|a'(p)|=\sup\limits_{p\in[0,T]}\int_{0}^{\infty}g(s)g''(p+s)ds\leq g''(0)\int_{0}^{\infty}g(s)ds<\infty.
	\end{equation*}
\end{example}
\subsection{Existence of the limiting moving average}
In this subsection we answer the following question: Does our limiting process have a moving average representation?
\\ Indeed, we only know that there is a limiting process such that it is a stationary centred Gaussian process with covariance structure given by $\int_{-\infty}^{\min(t,s)}g(t-r)g(s-r)dr$ for $t,s\in[0,T]$ where $g$ is a continuous function such that $g(x)=0$ for $x<0$. The answer is positive and it is given by the following proposition. 
\begin{proposition}
	Let $Y_{t}$ be a stationary Gaussian process with covariance $\int_{-\infty}^{\min(t,s)}g(t-r)g(s-r)dr$ for $t,s\in\mathbb{R}$. Assume that 
	\begin{equation}\label{ass}
	\int_{0}^{\infty}\int_{0}^{\infty}g(s+h)g(s)dsdh<\infty,\quad\text{and}\quad\int_{\mathbb{R}}\frac{\log(f(u))}{1+u^{2}}du>-\infty,
	\end{equation}
	where $f$ is the spectral density. Then we have that the limiting object can assume a moving average representation
	\begin{equation*}
	\left\{Y_{t}\right\}_{t\in\mathbb{R}}\stackrel{d}{=}\left\{\int_{-\infty}^{t}\bar{g}(t-s)d\bar{B}_{s}\right\}_{t\in\mathbb{R}},
	\end{equation*}
	where the integral is well defined since $\bar{B}$ is a one dimensional Brownian motion and $\bar{g}\in L^{2}(\mathbb{R})$ with $\bar{g}(x)=0$ for $x< 0$.
\end{proposition}
\begin{proof}
	First, we have that
	\begin{equation*}
	\lim\limits_{t\rightarrow0}\mathbb{E}[(Y_{t}-Y_{0})^{2}]=2\mathbb{E}[Y_{0}^{2}]-2\lim\limits_{t\rightarrow0}\mathbb{E}[Y_{t}Y_{0}]=2\int_{-\infty}^{0}g(-r)^{2}dr-2\lim\limits_{t\rightarrow0}\int_{-\infty}^{0}g(t-r)g(-r)dr=0,
	\end{equation*}
	by approximating $g$ with continuous functions with compact support (see Remark 2.4 of \cite{Cheridito2004}). Moreover, from
	\begin{equation*}
	\int_{0}^{\infty}\int_{0}^{\infty}g(s+h)g(s)dsdh<\infty,
	\end{equation*}
	we know that $Y_{t}$ has an absolutely continuous spectral distribution (see \cite{Doob1990}, page 532). Then, using the second part of assumption $(\ref{ass})$ together with \cite{Karhunen1950} Satz 5 we conclude the proof.
\end{proof}
Notice that the spectral density $f$ in our case is given by:
\begin{equation*}
f(u)=\int_{0}^{\infty}\int_{0}^{\infty}g(s+h)g(s)dse^{-2\pi iuh}dh,
\end{equation*}
and since we are in the real valued framework, it reduces to
\begin{equation*}
f(u)=\int_{0}^{\infty}\int_{0}^{\infty}g(s+h)g(s)ds\cos(2\pi uh)dh.
\end{equation*}
Further, in terms of the trawl function we have
\begin{equation*}
\int_{0}^{\infty}\int_{0}^{\infty}g(s+h)g(s)dsdh=\int_{0}^{\infty}\int_{0}^{\infty}a(s+h)dsdh,\quad\textnormal{and}\quad f(u)=\int_{0}^{\infty}\int_{0}^{\infty}a(s+h)ds\cos(2\pi uh)dh.
\end{equation*}
\section{Proofs\label{sec:Proofs}}

Through all our proofs the non-random positive constants will be denoted
by the generic symbol $C>0$, and they may change from line to line.
Additionally, for simplicity and without loss of generality, we may
and do assume that $\mathbb{E}(L^{\prime})=0$ and $Var(L^{\prime})=1$
in such a way that $\varGamma_{X}(h)=\int_{h}^{\infty}a(s)\mathrm{d}s$,
for $h\geq0$. We note that below we will use the notation $T_{n}=n\Delta_{n}$

\subsection{Technical lemmas}

We start by analyzing the variance of $S_{m}^{\Delta}$.

\begin{lemma}\label{lemmavarianceSm}

Suppose that $\mathbb{E}(\left|L^{\prime}\right|^{2})<\infty$. Then
\begin{equation}
Var(S_{m}^{\Delta})=\frac{2}{\Delta^{2}}\int_{0}^{m\Delta}\int_{0}^{r}\varGamma_{X}(s)\mathrm{d}s\mathrm{d}r+\mathrm{O}(m),\,\,\,m\in\mathbb{N},\Delta>0.\label{approximationvarS}
\end{equation}
Furthermore, if $\Delta m\rightarrow\mu\in[0,+\infty]$ as $\Delta\downarrow0$
and $m\uparrow\infty$, then
\begin{description}
\item [{i.}] If $\mu=0$, then 
\[
\frac{1}{m^{2}}Var(S_{m}^{\Delta})\rightarrow\varGamma_{X}(0).
\]
\item [{ii.}] If $0<\mu<\infty$, then 
\[
\Delta^{2}Var(S_{m}^{\Delta})\rightarrow2\int_{0}^{\beta}\int_{0}^{r}\varGamma_{X}(s)\mathrm{d}s\mathrm{d}r.
\]
\item [{iii.}] If $\mu=+\infty$, then
\begin{description}
\item [{a)}] If $\int_{0}^{\infty}\int_{r}^{\infty}a(s)\mathrm{d}s\mathrm{d}r<\infty$,
then $Var(S_{m}^{\Delta})\sim\int_{\mathbb{R}}\varGamma_{X}(s)\mathrm{d}s\frac{m}{\Delta}$,
as $\Delta\downarrow0$ and $m\uparrow\infty$.
\item [{b)}] If $\int_{0}^{\infty}\int_{r}^{\infty}a(s)\mathrm{d}s\mathrm{d}r=+\infty$
assume in addition that $a\in\mathrm{RV}_{\alpha}^{\infty}$ with
$1<\alpha<2$. Then, $Var(S_{m}^{\Delta})\sim c_{\alpha}Var(L')a(m\Delta)m^{3}\Delta$,
$c_{\alpha}=\frac{2}{(\alpha-1)(2-\alpha)(3-\alpha)}$.
\end{description}
\end{description}
\end{lemma}

\begin{proof}We have that 
\begin{equation}
Var(S_{m}^{\Delta})=m\varGamma_{X}(0)+2\sum_{i=1}^{m-1}\sum_{j=1}^{i}\varGamma_{X}(j\Delta).\label{varS_}
\end{equation}
Now, 
\begin{align*}
R(m,\Delta):=\frac{1}{\Delta^{2}}\left\{ \Delta^{2}\sum_{i=1}^{m-1}\sum_{j=1}^{i}\varGamma_{X}(j\Delta)-\int_{0}^{m\Delta}\int_{0}^{r}\varGamma_{X}(s)\mathrm{d}s\mathrm{d}r\right\}  &
\end{align*}
\begin{align*}
& =R_{1}(m,\Delta)+R_{2}(m,\Delta)+R_{3}(m,\Delta)+R_{4}(m,\Delta),
\end{align*}
where 
\begin{align*}
R_{1}(m,\Delta) & :=\frac{1}{\Delta^{2}}\sum_{i=1}^{m-1}\sum_{j=1}^{i}\int_{i\Delta}^{(i+1)\Delta}\int_{j\Delta}^{(j+1)\Delta}\left[\varGamma_{X}(j\Delta)-\varGamma_{X}(s)\right]\mathrm{d}s\mathrm{d}r;\\
R_{2}(m,\Delta) & :=-\frac{1}{\Delta^{2}}\int_{0}^{\Delta}\int_{0}^{r}\varGamma_{X}(s)\mathrm{d}s\mathrm{d}r;\\
R_{3}(m,\Delta) & :=\frac{1}{\Delta^{2}}\sum_{i=1}^{n-1}\int_{i\Delta}^{(i+1)\Delta}\int_{r}^{(i+1)\Delta}\varGamma_{X}(s)\mathrm{d}s\mathrm{d}r;\\
R_{4}(m,\Delta) & :=-\frac{m}{\Delta}\int_{0}^{\Delta}\varGamma_{X}(s)\mathrm{d}s.
\end{align*}
From (\ref{TrawlACF}), $\varGamma_{X}$ is non-increasing on $\mathbb{R}^{+},$
with derivative $-Var(L')a$. Therefore 
\begin{align*}
\left|R_{1}(m,\Delta)\right| & \leq C\frac{1}{\Delta}\sum_{i=1}^{m-1}\sum_{j=1}^{i}\int_{i\Delta}^{(i+1)\Delta}\int_{j\Delta}^{(j+1)\Delta}a(j\Delta)\mathrm{d}s\mathrm{d}r\\
 & =C\frac{1}{\Delta}\sum_{i=1}^{m-1}\sum_{j=1}^{i}\int_{i\Delta}^{(i+1)\Delta}\int_{(j-1)\Delta}^{j\Delta}a(j\Delta)\mathrm{d}s\mathrm{d}r\\
 & \leq C\frac{1}{\Delta}\sum_{i=1}^{m-1}\sum_{j=1}^{i}\int_{i\Delta}^{(i+1)\Delta}\int_{(j-1)\Delta}^{j\Delta}a(s)\mathrm{d}s\mathrm{d}r\\
 & \leq C\frac{1}{\Delta}\int_{0}^{m\Delta}\int_{0}^{r}a(s)\mathrm{d}s\mathrm{d}r.
\end{align*}
In a similar way, we obtain that 
\begin{align*}
\left|R_{2}(m,\Delta)\right| & +\left|R_{3}(m,\Delta)\right|\leq\frac{1}{\Delta}\int_{0}^{\Delta m}\varGamma_{X}(r)\mathrm{d}r.
\end{align*}
All above implies that
\begin{align*}
R(m,\Delta) & \leq Cm\varGamma_{X}(0).
\end{align*}
This estimate together with (\ref{varS_}) give (\ref{approximationvarS}).

Now assume that $\Delta m\rightarrow\beta\in[0,+\infty]$. i., ii.
and part a) of iii. follow immediately by (\ref{approximationvarS})
and the Dominated Convergence Theorem. Therefore, for the rest of
the proof we will assume that $\Delta m\rightarrow+\infty$ and that
$a\in\mathrm{RV}_{\alpha}^{\infty}$ in which $1<\alpha<2$. By KT
we get that 
\[
\frac{2}{\Delta^{2}}\int_{0}^{m\Delta}\int_{0}^{r}\varGamma_{X}(s)\mathrm{d}s\mathrm{d}r\sim c_{\alpha}a(m\Delta)m^{3}\Delta,\,\,\,\text{as }\Delta m\rightarrow+\infty.
\]
Since $a\in\mathrm{RV}_{\alpha}^{\infty}$, it admits the representation
$a(x)=x^{-\alpha}l(x)$, with $l$ a slowly varying function at $\infty$.
Thus,
\[
a(m\Delta)m^{2}\Delta=(m\Delta)^{2-\alpha}l(m\Delta)\Delta^{-1}\rightarrow+\infty,
\]
where we have used that for any slowly varying function $l(x)x^{\rho}\rightarrow+\infty$
as $x\uparrow\infty$ whenever $\rho>0$. Consequently, by (\ref{approximationvarS}),
we deduce that
\[
\frac{1}{a(m\Delta)m^{3}\Delta}\left|Var(S_{m}^{\Delta})-\frac{2}{\Delta}\int_{0}^{m\Delta}\int_{0}^{r}\varGamma_{X}(s)\mathrm{d}s\mathrm{d}r\right|\rightarrow0,\,\,\,\text{as }n\rightarrow\infty,
\]
which completes the proof.

\end{proof}

Next, we find a very useful decomposition for $S_{m}^{\Delta}$. For
any $\Delta>0$, let 
\[
\mathcal{\mathcal{P}}_{A}^{\Delta}(i,j):=\{(r,s):a(t_{j+1}-s)<r\leq a(t_{j}-s),t_{i-1}<s\leq t_{i}\},
\]
where $t_{i}=t_{i}(\Delta)=i\Delta$ with the convention that $t_{-1}=-\infty$.
It is clear that $\mathcal{\mathcal{P}}_{A}^{\Delta}(i,j)\cap\mathcal{\mathcal{P}}_{A}^{\Delta}(i^{\prime},j^{\prime})$
whenever either $i\neq i^{\prime}$ or $j\neq j^{\prime}$ for $i=0,\ldots,m-1$
and $j=0,1\ldots,m-2$, $j\geq i$. Moreover,
\begin{equation}
Leb\left\{ A_{k\Delta}\setminus\bigcup_{i=0}^{k}\bigcup_{j=k}^{\infty}\mathcal{\mathcal{P}}_{A}^{\Delta}(i,j)\right\} =0,\label{partitionAn}
\end{equation}
and
\begin{equation}
Leb(\mathcal{\mathcal{P}}_{A}^{\Delta}(i,i+j))=\begin{cases}
\int_{t_{j}}^{t_{j+1}}a(s)\mathrm{d}s & \textrm{ if }i=0,j\geq0;\\
\int_{t_{j}}^{t_{j+1}}[a(s)-a(s+\Delta)]\mathrm{d}s & \textrm{ if }i=1,\ldots,m-2,j<m-1-i.
\end{cases}\label{lebpartition}
\end{equation}
Based on these observations, the following result is obvious.

\begin{lemma}Let $\chi_{i,j}^{\Delta}:=L(\mathcal{\mathcal{P}}_{A}^{\Delta}(i,j))-\mathbb{E}(L(\mathcal{\mathcal{P}}_{A}^{\Delta}(i,j))).$
Then, almost surely
\begin{align}
S_{m}^{\Delta} & =\frac{1}{\Delta}\sum_{i=0}^{m-1}\sum_{j=i}^{m-1}t_{j-i+1}\chi_{i,j}^{\Delta}+\frac{1}{\Delta}\sum_{i=0}^{m-1}(t_{m}-t_{i})\zeta_{i,m}^{\Delta}\label{eq:decompSn}\\
 & =S_{m}^{\Delta,1}+S_{m}^{\Delta,2}+S_{m}^{\Delta,3}+S_{m}^{\Delta,4},
\end{align}
where $\zeta_{i,m}^{\Delta}:=\sum_{j=m}^{\infty}\chi_{i,j}^{\Delta}$
and 
\begin{align*}
S_{m}^{\Delta,1} & :=\frac{1}{\Delta}\sum_{i=1}^{m-1}\sum_{j=1}^{m-i}t_{j}\chi_{i,j+i-1}^{\Delta};\,\,\,S_{m}^{\Delta,2}:=\frac{1}{\Delta}\sum_{i=1}^{m-1}(t_{m}-t_{i})\zeta_{i,m}^{\Delta};\\
S_{m}^{\Delta,3} & :=\frac{1}{\Delta}\sum_{j=0}^{m-1}t_{j+1}\chi_{0,j}^{\Delta};\,\,\,S_{m}^{\Delta,4}:=\frac{t_{m}}{\Delta}\zeta_{0,m}^{\Delta}.
\end{align*}
\end{lemma}

When $\beta=+\infty$, it turns out that in the short memory case
$S_{m}^{\Delta,1}$ dominates the asymptotics.

\begin{lemma}\label{leadingtermshortmemory}

Let $m_{n}\in\mathbb{N}$ be such that $m_{n}\uparrow\infty$, $\Delta_{n}m_{n}\rightarrow\infty$
and $\Delta_{n}\downarrow0$, as $n\rightarrow\infty$. Suppose that
$\mathbb{E}(\left|L^{\prime}\right|^{2})<\infty$ and that $\int_{0}^{\infty}\int_{r}^{\infty}a(s)\mathrm{d}s\mathrm{d}r<\infty$.
Then 
\[
S_{m_{n}}^{\Delta_{n}}=S_{m_{n}}^{\Delta_{n},1}+\mathrm{o}_{\mathbb{P}}\left(\sqrt{\frac{m_{n}}{\Delta_{n}}}\right).
\]

\end{lemma}

The proof of Lemma \ref{leadingtermshortmemory} heavily relies on
the next property.

\begin{lemma}\label{shortmemorylimitatinf}Let $f\geq0$ be an integrable
continuous function such that $\int_{0}^{\infty}\int_{x}^{\infty}f(s)\mathrm{d}s\mathrm{d}x<\infty$.
Then, as $x\rightarrow+\infty$ 
\[
x\int_{x}^{\infty}f(s)\mathrm{d}s\rightarrow0,\,\,\text{and}\,\,\frac{1}{x}\int_{0}^{x}s^{2}f(s)\mathrm{d}s\rightarrow0.
\]

\end{lemma}

\begin{proof}If $f\equiv0$ a.e. the result is trivial, so assume
that $f>0$. For $x\geq0$, put $F(x):=\int_{x}^{\infty}f(s)\mathrm{d}s$.
Integration by parts gives that 
\[
\int_{0}^{x}F(s)\mathrm{d}s=xF(x)+\int_{0}^{x}sf(s)\mathrm{d}s.
\]
In view that $f>0$ and $\int_{0}^{\infty}F(s)ds<\infty$, the Dominated
Convergence Theorem guarantees that 
\[
0\leq\lim_{x\rightarrow\infty}\int_{0}^{x}sf(s)\mathrm{d}s=\int_{0}^{\infty}sf(s)\mathrm{d}s\leq\int_{0}^{\infty}F(s)\mathrm{d}s<\infty.
\]
This shows in particular that the following limit exists 
\[
\infty>\ell=\lim_{x\rightarrow\infty}xF(x)=\int_{0}^{\infty}F(s)\mathrm{d}s-\int_{0}^{\infty}sf(s)\mathrm{d}s\geq0.
\]
Observe that if $\ell>0$, then, as $x\rightarrow+\infty$, $1/\left(x\int_{x}^{\infty}f(s)ds\right)\rightarrow1/\ell$.
Thus, if $\ell>0$, we could find $x_{0}>0$ such that for all $x>x_{0}$
\[
\frac{1}{x}<C\int_{x}^{\infty}f(s)\mathrm{d}s,
\]
which contradicts that $\int_{0}^{\infty}\int_{x}^{\infty}f(s)\mathrm{d}s\mathrm{d}x<\infty$.
Hence, $\ell=0$ as required. 

To show the last part, observe first that when $\int_{0}^{\infty}\int_{y}^{\infty}\int_{x}^{\infty}f(s)\mathrm{d}s\mathrm{d}x\mathrm{d}y<\infty$,
an analogous argument as above shows that 
\[
0\leq\int_{0}^{\infty}sF(s)\mathrm{d}s\leq\int_{0}^{\infty}\int_{s}^{\infty}F(x)\mathrm{d}x\mathrm{d}s<\infty.
\]
Therefore, from the first part of the proof, as $x\rightarrow+\infty$
\begin{align*}
\frac{1}{x}\int_{0}^{x}s^{2}f(s)ds & =-xF(x)+\frac{2}{x}\int_{0}^{x}F(s)sds\rightarrow0.
\end{align*}
Now suppose that $\int_{0}^{\infty}\int_{y}^{\infty}\int_{x}^{\infty}f(s)\mathrm{d}s\mathrm{d}x\mathrm{d}y=+\infty$
and put $\bar{F}(x):=\int_{0}^{x}F(s)\mathrm{d}s$. Clearly as $\int_{0}^{x}\bar{F}(s)\mathrm{d}s\rightarrow+\infty$
$x\rightarrow+\infty$, and for all $x\geq0$ 
\begin{equation}
\frac{1}{x}\int_{0}^{x}s^{2}f(s)ds=-xF(x)+2\left[\bar{F}(x)-\frac{1}{x}\int_{0}^{x}\bar{F}(s)ds\right].\label{intbypartslemmashort}
\end{equation}
Moreover, by L'Hospital's Rule and the continuity of $f$ we have
that
\[
\frac{1}{x}\int_{0}^{x}\bar{F}(s)\mathrm{d}s\rightarrow\int_{0}^{\infty}F(s)\mathrm{d}s,
\]
which applied to (\ref{intbypartslemmashort}) concludes the proof.\end{proof}

\begin{proof}[Proof of Lemma \ref{leadingtermshortmemory}]Since $L$
is independently scattered, we get by (\ref{lebpartition}) that for
any $m\in\mathbb{N}$ and $\Delta>0$

\begin{align*}
Var(S_{m}^{\Delta,2}) & =\frac{1}{\Delta^{2}}\sum_{i=1}^{m-1}\int_{t_{i-1}}^{t_{i}}(t_{m-1}-t_{i-1})^{2}a(t_{m-1}-s)\mathrm{d}s;\\
Var(S_{m}^{\Delta,3}) & =\frac{1}{\Delta^{2}}\int_{0}^{\Delta m}s^{2}a(s)\mathrm{d}s+\frac{2}{\Delta^{2}}\sum_{j=0}^{m-1}\int_{t_{j}}^{t_{j+1}}\int_{s}^{t_{j+1}}r\mathrm{d}ra(s)\mathrm{d}s;\\
Var(S_{m}^{\Delta,4}) & =m^{2}\int_{\Delta m}^{\infty}a(s)ds.
\end{align*}
Moreover, in view that the trawl function is non-negative continuous
and such that $\int_{0}^{\infty}\int_{x}^{\infty}a(s)\mathrm{d}s\mathrm{d}x<\infty$,
Lemma \ref{shortmemorylimitatinf} can be applied in order to obtain
that
\[
\frac{\Delta_{n}}{m_{n}}Var(S_{m_{n}}^{\Delta_{n},4})=\Delta_{n}m_{n}\int_{\Delta_{n}m_{n}}^{\infty}a(s)ds\rightarrow0.
\]
We proceed now to show that for every $m\in\mathbb{N}$ and $\Delta>0$

\begin{align}
\left|Var(S_{m}^{\Delta,2})-\frac{1}{\Delta^{2}}\int_{0}^{\Delta(m-1)}s^{2}a(s)\mathrm{d}s\right| & \leq C\frac{1}{\Delta}\int_{0}^{\Delta m}sa(s)\mathrm{d}s+\mathrm{O}(1);\label{approxvarS2}\\
\left|Var(S_{m}^{\Delta,3})-\frac{1}{\Delta^{2}}\int_{0}^{\Delta m}s^{2}a(s)\mathrm{d}s\right| & \leq C\frac{1}{\Delta}\int_{0}^{\Delta m}sa(s)\mathrm{d}s+\mathrm{O}(1).\label{approxvarS3}
\end{align}
Let $R(m,\Delta)=\sum_{i=1}^{m-1}\int_{t_{i-1}}^{t_{i}}(t_{m-1}-t_{i-1})^{2}a(t_{m-1}-s)\mathrm{d}s$
and $R^{\prime}(m,\Delta)=\sum_{j=0}^{m-1}\int_{t_{j}}^{t_{j+1}}\int_{s}^{t_{j+1}}r\mathrm{d}ra(s)\mathrm{d}s$.
Then,
\begin{align*}
\left|R(m,\Delta)-\int_{0}^{\Delta(m-1)}(t_{m-1}-s)^{2}a(t_{m-1}-s)\mathrm{d}s\right| & \leq C\Delta\sum_{i=1}^{m-1}\int_{t_{i-1}}^{t_{i}}(t_{m-1}-t_{i-1})a(t_{m-1}-s)\mathrm{d}s\\
 & \leq C\Delta\int_{0}^{\Delta m}sa(s)\mathrm{d}s+C\Delta^{2}\int_{0}^{\Delta m}a(s)\mathrm{d}s,
\end{align*}
which is exactly (\ref{approxvarS2}). In a similar way, we see that
\begin{align*}
\left|R^{\prime}(m,\Delta)\right| & \leq2\Delta\sum_{j=0}^{m-1}\int_{t_{j}}^{t_{j+1}}t_{j+1}a(s)\mathrm{d}s\\
 & \leq2\Delta\int_{0}^{\Delta m}sa(s)\mathrm{d}s+2\Delta^{2}\int_{0}^{\Delta m}a(s)\mathrm{d}s.
\end{align*}
Relation (\ref{approxvarS3}) now follows easily from this. Finally,
note that from (\ref{approxvarS2}), (\ref{approxvarS3}) and Lemma
\ref{shortmemorylimitatinf}, it follows that for $l=2,3$
\begin{align*}
\frac{\Delta_{n}}{m_{n}}Var(S_{m_{n}}^{\Delta_{n},l}) & =\frac{1}{m_{n}\Delta_{n}}\int_{0}^{\Delta_{n}m_{n}}s^{2}a(s)\mathrm{d}s+\mathrm{o}(1)\rightarrow0,\,\,\,n\rightarrow\infty,
\end{align*}
completing this the proof.\end{proof}

We proceed now to find some estimates for the characteristic function
of $S_{m}^{\Delta,l}$, for $l=3,\ldots,4$. For doing this, the following
result is essential and its proof follows the lines of the proof of
Proposition 3.6 in \cite{RajputRosinski89} as well as the well-known
inequality 
\[
\left|e^{\mathbf{i}zx}-1\right|\leq2\left(\left|zx\right|\mathbf{1}_{\left|zx\right|\leq1}+\mathbf{1}_{\left|zx\right|>1}\right).
\]

\begin{lemma}\label{lemmaaproxZ}

Let $\psi$ the characteristic exponent of an ID distribution with
mean $0$. Then $\psi$ is continuously differentiable and there is
a constant $C>0$ depending only on $(\gamma,b,\nu)$ such that 
\begin{equation}
\left|\psi(z)\right|\leq b^{2}\left|z\right|^{2}+C\int_{\mathbb{R}}(1\land\left|xz\right|^{2})\nu(\mathrm{d}x),\,\,\,z\in\mathbb{R};\label{bounderchexp-1}
\end{equation}
\begin{equation}
\left|\psi^{\prime}(z)\right|\leq b^{2}\left|z\right|+C\int_{\mathbb{R}}(1\land\left|xz\right|)\left|x\right|\nu(\mathrm{d}x),\,\,\,z\in\mathbb{R}.\label{bounderchexp}
\end{equation}

\end{lemma}

\begin{lemma}\label{lemmaapproxchf}Suppose that $\mathbb{E}\left(\left|L^{\prime}\right|^{2}\right)<\infty$
and let
\begin{align*}
I_{m}^{\Delta,1}(z) & :=\int_{0}^{\Delta m}(\Delta m-s)\psi\left(\frac{s}{\Delta}z\right)\left[\frac{a(s)-a(s+\Delta)}{\Delta}\right]\mathrm{d}s;\\
I_{m}^{\Delta,2}(z) & :=\int_{0}^{\Delta m}\psi\left(\frac{s}{\Delta}z\right)a(s)\mathrm{d}s.
\end{align*}
Then the following estimates hold
\begin{align*}
\left|\mathcal{C}\left(z\ddagger S_{n}^{\Delta_{n},1}\right)-I_{m}^{\Delta,1}(z)\right| & \leq C\left|z\right|^{2}\left(m+\frac{1}{\Delta}\right)\int_{0}^{t_{m}}s\left[\frac{a(s)-a(s+\Delta)}{\Delta}\right]\mathrm{d}s\\
 & +C\Delta\int_{0}^{\Delta m}\left|\psi\left(\frac{s}{\Delta}z\right)\right|\left[\frac{a(s)-a(s+\Delta)}{\Delta}\right]\mathrm{d}s;\\
\left|\mathcal{C}\left(z\ddagger S_{n}^{\Delta_{n},2}\right)-I_{m}^{\Delta,2}(z)\right|+\left|\mathcal{C}\left(z\ddagger S_{n}^{\Delta_{n},3}\right)-I_{m}^{\Delta,3}(z)\right| & \leq C\frac{\left|z\right|^{2}}{\Delta}\int_{0}^{\Delta m}sa(s)\mathrm{d}s.
\end{align*}

\end{lemma}

\begin{proof}Recall that we assume that $L$ is centered. By the
independent scattered property of $L$ it follows from 
\begin{align}
\mathcal{C}\left(z\ddagger S_{n}^{\Delta_{n},1}\right)= & \sum_{j=1}^{m-1}\int_{t_{j}}^{t_{j+1}}(t_{m}-s)\psi\left(\frac{t_{j}}{\Delta}z\right)\left[\frac{a(s)-a(s+\Delta)}{\Delta}\right]\mathrm{d}s\label{chfncS1}\\
 & +\sum_{j=1}^{m-1}\int_{t_{j}}^{t_{j+1}}(s-t_{j})\psi\left(\frac{t_{j}}{\Delta}z\right)\left[\frac{a(s)-a(s+\Delta)}{\Delta}\right]\mathrm{d}s\nonumber; \\
\mathcal{C}\left(z\ddagger S_{n}^{\Delta_{n},2}\right) & =\sum_{i=1}^{m-1}\int_{t_{i-1}}^{t_{i}}\psi\left(\frac{t_{m}-t_{i}}{\Delta}z\right)a(t_{m}-s)\mathrm{d}s;\label{chfncS2}\\
\mathcal{C}\left(z\ddagger S_{n}^{\Delta_{n},3}\right) & =\sum_{j=1}^{m-1}\int_{t_{j}}^{t_{j+1}}\psi\left(\frac{t_{j+1}}{\Delta}z\right)a(s)\mathrm{d}s.\label{chfncS3}
\end{align}
The claimed estimates are easily obtained by noting that from Lemma
\ref{lemmaapproxchf} and the Mean Value Theorem
\[
\left|\psi\left(\frac{t_{j}}{\Delta}z\right)-\psi\left(\frac{s}{\Delta}z\right)\right|\leq C\frac{\left|z\right|^{2}}{\Delta}s,\,\,\,t_{j}\leq s\leq t_{j+1},j\in\mathbb{N}.
\]

\end{proof}

\subsection{Proof of Propositions \ref{propmufinite} and \ref{porpndeltnzero}}

\begin{proof}[Proof of Proposition \ref{propmufinite}]

For simplicity we will assume that $\mu=1$. Following the reasoning
in Section 3 in \cite{BasseBN11}, we can always find a measurable
modification of $X$, so without loss of generality we may and do
assume that $X$ is measurable and almost surely $\int_{0}^{t}X_{s}^{2}\mathrm{d}s<\infty$,
for all $t\geq0$. Thus, using the well known bound $\left(\sum_{i=1}^{d}\left|x_{i}\right|\right)^{2}\leq d\sum_{i=1}^{d}\left|x_{i}\right|^{2}$
and Jensen's inequality, we see that for any $V>0$ and $t\leq V$
\[
\left|\Delta_{n}\mathbf{S}_{t}^{n}-\int_{0}^{\left[nt\right]\Delta_{n}}X_{s}\mathrm{d}s\right|^{2}\leq Vn\Delta_{n}\sum_{i=0}^{\left[nt\right]-1}\int_{t_{i}}^{t_{i+1}}\left|X_{t_{i}}-X_{s}\right|^{2}\mathrm{d}s\leq C\sum_{i=0}^{\left[nV\right]-1}\int_{t_{i}}^{t_{i+1}}\left|X_{t_{i}}-X_{s}\right|^{2}\mathrm{d}s,
\]
where we have used that $n\Delta_{n}$ is bounded. From this estimate
we deduce that as $n\rightarrow\infty$
\begin{align*}
\mathbb{E}\left(\sup_{0\leq t\leq T}\left|\Delta_{n}\mathbf{S}_{t}^{n}-\int_{0}^{\left[nt\right]\Delta_{n}}X_{s}\mathrm{d}s\right|^{2}\right) & \leq C\sum_{i=0}^{\left[nV\right]-1}\int_{t_{i}}^{t_{i+1}}\int_{0}^{s-t_{i}}a(r)\mathrm{d}r\mathrm{d}s\\
 & \leq Ca(0)(\left[nV\right]\Delta_{n})\Delta_{n}\rightarrow0.
\end{align*}
The result now follows by observing that in view that $\left|t-\left[nt\right]\Delta_{n}\right|\leq\Delta_{n}+V\left|1-\Delta_{n}n\right|$
and that 
\begin{align*}
\left|\int_{0}^{t}X_{s}\mathrm{d}s-\int_{0}^{\left[nt\right]\Delta_{n}}X_{s}\mathrm{d}s\right|^{2} & \leq\left|t-\left[nt\right]\Delta_{n}\right|\int_{0}^{C}X_{s}^{2}\mathrm{d}s.
\end{align*}

\end{proof}

\begin{proof}[Proof of Proposition \ref{porpndeltnzero}]Plainly,
from (\ref{partitionAn}) 
\begin{equation}
\frac{\left[nt\right]}{n}X_{0}-\frac{1}{n}S_{[nt]}^{\Delta_{n},4}=\frac{\left[nt\right]}{n}\sum_{j=0}^{\left[nt\right]-1}\chi_{0,j}^{\Delta_{n}},\label{deltanzerodecomposition1}
\end{equation}
which in view of (\ref{lebpartition}) implies that
\[
\mathcal{C}\left(z\ddagger\left(\frac{1}{n}S_{[nt]}^{\Delta_{n},4}-\frac{\left[nt\right]}{n}X_{0}\right)\right)=\psi\left(\frac{\left[nt\right]}{n}z\right)\int_{0}^{\left[nt\right]\Delta_{n}}a(s)\mathrm{d}s\rightarrow0.
\]
Therefore, thanks to (\ref{eq:decompSn}) we only need to check that
for $l=1,2,3$, $\frac{1}{n}S_{n}^{\Delta_{n},l}\overset{\mathbb{P}}{\rightarrow}0$.
To see this observe that from equations (\ref{chfncS1})-(\ref{chfncS3}),
the continuity of $\psi$ 
\begin{align*}
\left|\mathcal{C}\left(z\ddagger\frac{1}{n}S_{n}^{\Delta_{n},1}\right)\right| & \leq C\int_{0}^{t_{n}}\left(t_{n}-s+\Delta_{n}\right)\left|\frac{a(s)-a(s+\Delta_{n})}{\Delta_{n}}\right|\mathrm{d}s\leq Ct_{n}^{2}\rightarrow0,
\end{align*}
and
\[
\left|\mathcal{C}\left(z\ddagger\frac{1}{n}S_{n}^{\Delta_{n},2}\right)+\mathcal{C}\left(z\ddagger S_{n}^{\Delta_{n},3}\right)\right|\leq C\int_{0}^{t_{n}}a(s)\mathrm{d}s\rightarrow0,
\]
where we have further used that $a$ is continuously differentiable
in a neighborhood of $0$. This completes the proof.\end{proof}

\subsection{Proof of Theorem \ref{thmndeltnzero}}

Our proof in this case relies heavily on the asymptotic behavior of
the L\'{e}vy measure of $L$ around $0$. It is worth noting that if $L$
is deterministic, then almost surely $Z_{t}^{n}\equiv0$, so by the
L\'{e}vy-It\^{o} decomposition of L\'{e}vy bases (see \cite{Ped03}), in our proof
we will always assume that $\gamma=\int_{\left|x\right|\leq1}x\nu(\mathrm{d}x)$
or $\gamma=0$, depending whether $\int_{\mathbb{R}}(1\land\left|x\right|)\nu(\mathrm{d}x)<\infty$
or not. In this situation, under Assumption \ref{assumptionmu0},
Theorem 2 in \cite{IvanovJ18} establishes that as $\varepsilon\rightarrow0$
\begin{equation}
\varepsilon\psi(\varepsilon^{-1/\beta}z)\rightarrow\psi_{\beta}(z)=\begin{cases}
-\frac{1}{2}b^{2}z^{2} & \text{if }b>0\text{ and }\beta=2;\\
\psi(z;\beta,K_{+}\beta,K_{-}\beta,\tilde{\gamma}) & \text{under iii.}\text{ and }0<\beta<2;
\end{cases}\label{chfctat0}
\end{equation}
where $\psi(\cdot;\beta,K_{+}\beta,K_{-}\beta,\tilde{\gamma})$ as
in (\ref{chfnctstrictlystabe}). Note that the convergence takes place
uniformly in compacts. The proof is divided in several steps: In the
first step we show that $S_{n}^{\Delta_{n},1}=\mathrm{o}_{\mathbb{P}}(nT_{n}{}^{1/\beta})$.
In the second step we argue that $L$ can be assumed to be strictly
$\beta$-stable. Finally, we show that \textbf{i. and ii. }hold. 

\begin{proof}[Step 1: $S_{n}^{\Delta_{n},1}=\mathrm{o}_{\mathbb{P}}(n(n\Delta_{n})^{1/\beta})$]Assume
that (\ref{chfctat0}) holds. and put
\begin{align*}
A_{n}^{\prime}(z):= & \sum_{j=1}^{n-1}\int_{t_{j}}^{t_{j+1}}(T_{n}-s)\psi_{\beta}\left(\frac{t_{j}}{T_{n}{}^{1+1/\beta}}z\right)\left[\frac{a(s)-a(s+\Delta_{n})}{\Delta_{n}}\right]\mathrm{d}s\\
 & +\sum_{j=1}^{n-1}\int_{t_{j}}^{t_{j+1}}(s-t_{j})\psi_{\beta}\left(\frac{t_{j}}{T_{n}{}^{1+1/\beta}}z\right)\left[\frac{a(s)-a(s+\Delta_{n})}{\Delta_{n}}\right]\mathrm{d}s.
\end{align*}
The $C^{1}$ property of $a$ and the fact that $0\leq t_{j}/T_{n}\leq1$
then lead us to
\begin{align*}
\left|\mathcal{C}\left(z\ddagger\frac{1}{nT_{n}^{1/\beta}}S_{n}^{\Delta_{n},2}\right)-A_{n}^{\prime}(z)\right| & \leq C\left(\sup_{\left|u\right|\leq\left|z\right|}T_{n}\left|\psi_{\beta}\left(\frac{u}{T_{n}{}^{1/\beta}}\right)-\psi\left(\frac{u}{T_{n}{}^{1/\beta}}\right)\right|\right)(T_{n}+\Delta_{n})\rightarrow0.
\end{align*}
Similarly
\[
\left|A_{n}^{\prime}(z)\right|\leq C(T_{n}+\Delta_{n})\rightarrow0,\,\,\,n\rightarrow\infty,
\]
where we have also used the fact that $\psi_{\beta}$ is strictly
stable and continuous. This is enough for the negligibility of $\frac{1}{n(n\Delta_{n})^{1/\beta}}S_{n}^{\Delta_{n},1}$.

\end{proof}

\begin{proof}[Step 2: An approximation]In this step we assume that
(\ref{chfctat0}) for some $0<\beta\leq2$. From (\ref{eq:decompSn}),
(\ref{deltanzerodecomposition1}) the previous step, we have that
\begin{equation}
Z_{t}^{n}=U_{t}^{n}-\hat{U}_{t}^{n}+\mathrm{o}_{\mathbb{P}}(T_{n}^{1/\beta}),\,\,t\geq0,\label{decompUUtilde}
\end{equation}
where $U_{t}^{n}:=\sum_{j=0}^{\left[nt\right]-1}\left(\frac{\left[nt\right]}{n}-\frac{j+1}{n}\right)\chi_{0,j}^{\Delta_{n}}$
and $\hat{U}_{t}^{n}:=\frac{1}{n}S_{\left[nt\right]}^{\Delta_{n},2}$.
Furthermore, $(U_{t}^{n,\beta},\hat{U_{t}}^{n,\beta})_{t\geq0}$ are
defined as $(U_{t}^{n},\hat{U_{t}}^{n})_{t\geq0}$ when we replace
$L$ by a homogeneous strictly $\beta$-stable distribution whose
seed has characteristic exponent given by $\psi_{\beta}$. We are
going to show that the f.d.d. of $(U^{n},\hat{U}^{n})$ are asymptotically
equivalent to those of $(U^{n,\beta},\hat{U}^{n,\beta})$. Indeed,
fix $q\in\mathbb{N}$, $\lambda_{1},\ldots,\lambda_{q}\in\mathbb{R}$
and $0=u_{0}<u_{1}<\cdots<u_{q}$ and note that
\begin{align*}
\sum_{l=1}^{q}\lambda_{l}U_{u_{l}}^{n} & =\sum_{j=0}^{\left[nq\right]-1}\theta_{j,n}\chi_{0,j}^{\Delta_{n}},\\
\sum_{l=1}^{q}\lambda_{l}\hat{U}_{u_{l}}^{n} & =\sum_{i=1}^{\left[nq\right]-1}\sum_{k=1}^{q}\left(\zeta_{i,\left[nu_{k}\right]}^{\Delta_{n}}-\zeta_{i,\left[nu_{k+1}\right]}^{\Delta_{n}}\right)\hat{\theta}_{i,k,n},
\end{align*}
where $\zeta_{i,\left[nu_{q+1}\right]}^{\Delta_{n}}:=0$, and
\begin{align*}
\theta_{j,n} & :=\sum_{l=1}^{q}\sum_{m=l}^{q}\lambda_{m}\mathbf{1}_{\left[nu_{l-1}\right]\leq j<\left[nu_{l}\right]}\left(\frac{\left[nu_{m}\right]}{n}-\frac{j+1}{n}\right),\\
\hat{\theta}_{i,k,n} & :=\sum_{m=1}^{k}\sum_{l=m}^{k}\lambda_{l}\mathbf{1}_{\left[nu_{m-1}\right]\leq i<\left[nu_{m}\right]}\left(\frac{\left[nu_{l}\right]}{n}-\frac{i}{n}\right).
\end{align*}
 Whence, from (\ref{lebpartition}) and (\ref{chfctat0}) as $n\rightarrow\infty$
\begin{align*}
\left|\mathcal{C}\left(z\ddagger\frac{1}{T_{n}^{1/\beta}}\sum_{l=1}^{q}\lambda_{l}U_{u_{l}}^{n}\right)-\mathcal{C}\left(z\ddagger\frac{1}{T_{n}^{1/\beta}}\sum_{l=1}^{q}\lambda_{l}U_{u_{l}}^{n,\beta}\right)\right| & \leq C\sup_{\left|u\right|\leq\left|C_{\lambda}z\right|}T_{n}\left|\psi_{\beta}\left(\frac{u}{T_{n}{}^{1/\beta}}\right)-\psi\left(\frac{u}{T_{n}{}^{1/\beta}}\right)\right|\rightarrow0,\\
\left|\mathcal{C}\left(z\ddagger\frac{1}{T_{n}^{1/\beta}}\sum_{l=1}^{q}\lambda_{l}\hat{U}_{u_{l}}^{n}\right)-\mathcal{C}\left(z\ddagger\frac{1}{T_{n}^{1/\beta}}\sum_{l=1}^{q}\lambda_{l}\hat{U}_{u_{l}}^{n,\beta}\right)\right| & \leq C\sup_{\left|u\right|\leq\left|C_{\lambda}z\right|}T_{n}\left|\psi_{\beta}\left(\frac{u}{T_{n}{}^{1/\beta}}\right)-\psi\left(\frac{u}{T_{n}{}^{1/\beta}}\right)\right|\rightarrow0.
\end{align*}
where $C_{\lambda}:=2u_{q}\sum_{l=1}^{q}\sum_{m=l}^{q}\left|\lambda_{m}\right|\mathbf{1}_{\left[nu_{l-1}\right]\leq j<\left[nu_{l}\right]}\geq\left|\theta_{j,n}\right|+\left|\hat{\theta}_{i,k,n}\right|$,
as claimed.

\end{proof}

\begin{proof}[Step 3: Proof of i. and ii.]We start by showing that
the f.d.d. distributions of $Z^{n}$converge to those stated in the
theorem. In the last part we show that the convergence in distribution
can strengthened to stable convergence.

Assume that $b>0$. In this case, in virtue of Step 2, we may and
do assume that $\gamma=0$ and $\nu\equiv0$. Accordingly, $Z^{n}$
is a centered Gaussian process satisfying (\ref{decompUUtilde}).
Thus, by the independent scattered property of $L$, the convergence
in i. is achieved whenever
\begin{equation}
\frac{1}{T_{n}}\mathbb{E}(U_{t}^{n}U_{u}^{n})\rightarrow\sigma^{2}\int_{0}^{t\land u}(t-r)(u-r)\mathrm{d}r;\,\,\frac{1}{T_{n}}\mathbb{E}(\hat{U}_{t}^{n}\hat{U}_{u}^{n})\rightarrow\sigma^{2}\int_{0}^{t\land u}(t-r)(u-r)\mathrm{d}r.\label{convergbnonull}
\end{equation}
To see that this is the case, take $t\geq u\geq0$. Then
\begin{align*}
\frac{1}{T_{n}}\mathbb{E}\left(U_{t}^{n}U_{u}^{n}\right) & =b^{2}\sum_{j=0}^{\left[nu\right]-1}\int_{j/n}^{(j+1)/n}\left(\frac{\left[nt\right]}{n}-\frac{j+1}{n}\right)\left(\frac{\left[nu\right]}{n}-\frac{j+1}{n}\right)a(T_{n}s)\mathrm{d}s.\\
\frac{1}{T_{n}}\mathbb{E}(\hat{U}_{t}^{n}\hat{U}_{u}^{n}) & =b^{2}\sum_{i=1}^{\left[nu\right]-1}\int_{(i-1)/n}^{i/n}\left(\frac{\left[nt\right]}{n}-\frac{i}{n}\right)\left(\frac{\left[nu\right]}{n}-\frac{i}{n}\right)a\left[T_{n}(\frac{\left[nt\right]}{n}-s)\right]\mathrm{d}s.
\end{align*}
(\ref{convergbnonull}) follows now an easy application of the Dominated
Convergence Theorem.

Suppose now that $b=0$ and Assumption \ref{assumptionmu0} holds,
such that (\ref{convergbnonull}). Therefore, by previous step, we
may and do assume that $L$ is strictly stable with characteristic
exponent $\psi_{\beta}$. Therefore, under the notation of Step 2,
the strict stability of $\psi_{\beta}$ results in
\begin{align*}
\mathcal{C}\left(z\ddagger\frac{1}{T_{n}^{1/\beta}}\sum_{l=1}^{q}\lambda_{l}U_{u_{l}}^{n}\right) & =\sum_{l=1}^{q}\sum_{j=\left[nu_{l-1}\right]}^{\left[nu_{l}\right]-1}\int_{j/n}^{(j+1)/n}\psi_{\beta}\left(\sum_{m=l}^{q}\lambda_{m}\left(\frac{\left[nu_{m}\right]}{n}-\frac{j+1}{n}\right)z\right)a(T_{n}s)\mathrm{d}s\\
 & \rightarrow a(0)\sum_{l=1}^{q}\int_{u_{l-1}}^{u_{l}}\psi_{\beta}\left(\sum_{m=l}^{q}\lambda_{m}\left(u_{m}-s\right)z\right)\mathrm{d}s,\,\,\,n\rightarrow\infty.
\end{align*}
Similarly, as $n\rightarrow\infty$
\begin{align*}
\mathcal{C}\left(z\ddagger\frac{1}{T_{n}^{1/\beta}}\sum_{l=1}^{q}\lambda_{l}\hat{U}_{u_{l}}^{n}\right)= & \sum_{x=1}^{q}\sum_{j=\left[nu_{x-1}\right]}^{\left[nu_{x}\right]-1}\int_{(i-1)/n}^{i/n}\psi_{\beta}\left(\sum_{l=x}^{q}\lambda_{m}\left(\frac{\left[nu_{l}\right]}{n}-\frac{i}{n}\right)z\right)a\left[T_{n}\left(\frac{[nu_{q}]}{n}-s\right)\right]\mathrm{d}s\\
 & +\sum_{x=1}^{q}\sum_{j=\left[nu_{x-1}\right]}^{\left[nu_{x}\right]-1}\sum_{k=x}^{q-1}\int_{(i-1)/n}^{i/n}\psi_{\beta}\left(\sum_{l=x}^{k}\lambda_{m}\left(\frac{\left[nu_{l}\right]}{n}-\frac{i}{n}\right)z\right)\\
 & \times\left(a\left[T_{n}\left(\frac{[nu_{k}]}{n}-s\right)\right]-a\left[T_{n}\left(\frac{[nu_{k+1}]}{n}-s\right)\right]\right)\mathrm{d}s\\
 & \rightarrow a(0)\sum_{x=1}^{q}\int_{u_{x-1}}^{u_{x}}\psi_{\beta}\left(\sum_{l=x}^{q}\lambda_{m}\left(u_{l}-s\right)z\right)\mathrm{d}s.
\end{align*}
We have therefore shown that the f.d.d. of $Z$ converge weakly to
those of stated in theorem. Therefore, in order to conclude the proof
it rests to verify that the convergence also take place stably and
the limit is independent of $L$. Let $B$ be a bounded Borel set.
Since for every $n\in\mathbb{N}$,$Z^{n}$ is$\mathcal{F}_{L}$-measurable,
thanks to Theorem 3.2 in \cite{HauslerLuschgy15}, it is sufficient
to show that
\begin{equation}
(\{Z_{u_{l}}\}_{l=1}^{q},L(B))\rightarrow\left(\{H_{u_{l}}\}_{l=1}^{q},L(B)\right),\label{stableconvmu0}
\end{equation}
and that for all $z_{1},\ldots,z_{q+1}\in\mathbb{R}$.
\begin{equation}
\mathcal{C}\left((z_{1},\ldots,z_{q+1})\ddagger\left(\{H_{u_{l}}\}_{l=1}^{q},L(B)\right)\right)=\mathcal{C}\left((z_{1},\ldots,z_{q})\ddagger\{H_{u_{l}}\}_{l=1}^{q}\right)+\mathcal{C}\left(z_{q+1}\ddagger L(B)\right).\label{stableconvmu02}
\end{equation}
Put 
\[
B_{n}=\bigcup_{j=0}^{[u_{q}n]}\mathcal{\mathcal{P}}_{A}^{\Delta_{n}}(0,j)\cup\bigcup_{i=1}^{[u_{q}n]}\bigcup_{j\geq i}\mathcal{\mathcal{P}}_{A}^{\Delta_{n}}(i,j).
\]
Then by (\ref{lebpartition}), $Leb(B\cap B_{n})\leq2a(0)[u_{q}n]\Delta_{n}\rightarrow0$,
meaning this that $L(B\cap B_{n})\overset{\mathbb{P}}{\rightarrow}0$.
Relations (\ref{stableconvmu0}) and (\ref{stableconvmu02}) are easily
obtained by decomposing $L(B)=L(B\cap B_{n})+L(B\backslash B_{n})$,
the preceeding observation, and an application of Slutsky's Theorem.

\end{proof}

\subsection{Proof of Theorem \ref{theoremmean}}

Here we show the validity of Theorem \ref{theoremmean}. The proof
will be divided into three steps. We first show the convergence of
the finite-dimensional distributions. Secondly, we verify that our
sequence is tight. We conclude by proving that the convergence is
also stable. Therefore, for the rest of this subsection we will let
Assumption \ref{assumptionshortmem} hold. We finally emphasize that
thanks to the Lévy-Itô decomposition of L\'{e}vy bases (see \cite{Ped03})
and Lemma \ref{lemmavarianceSm}, we may and do assume that $L$ has
no Gaussian component, i.e. $b=0$.

\begin{proof}[Step 1: Convergence of the f.d.d.]We start by verifying
that for any $m_{n}\in\mathbb{N}$ such that $m_{n}\uparrow\infty$,
$\Delta_{n}m_{n}\rightarrow\infty$ and $\Delta_{n}\downarrow0$,
as $n\rightarrow\infty$, it holds that
\begin{equation}
\sqrt{\frac{\Delta_{n}}{m_{n}}}S_{m_{n}}^{\Delta_{n}}\overset{d}{\rightarrow}\sigma_{a}N(0,1).\label{cltshortmemory}
\end{equation}
Observe first that thanks to Assumption \ref{assumptionshortmem},
$\int_{\mathbb{R}}\left|\varGamma_{X}(s)\right|\mathrm{d}s<\infty$,
and that for any $0\leq p<p_{0}$ the measure 
\[
\mu_{p,a}(\mathrm{ds})=\mathbf{1}_{s\geq0}s^{p}\mathrm{d}\left|a\right|(s),
\]
is finite. Moreover, from Lemma \ref{leadingtermshortmemory}, we
have that
\[
\sqrt{\frac{\Delta_{n}}{m_{n}}}S_{m_{n}}^{\Delta_{n}}=\sqrt{\frac{\Delta_{n}}{m_{n}}}S_{m_{n}}^{\Delta_{n},1}+\mathrm{o}_{\mathbb{P}}(1)=\frac{1}{\sqrt{\Delta_{n}m_{n}}}\sum_{j=1}^{m_{n}-1}t_{j}\xi_{j,n}+\mathrm{o}_{\mathbb{P}}(1),
\]
where the array
\[
\xi_{j,n}:=\sum_{i=1}^{m_{n}-j}\chi_{i,i+j-1}^{\Delta_{n}},\,\,\,j=1,\ldots,m_{n}-2,
\]
is centered and row-wise independent. Therefore, (\ref{cltshortmemory})
will hold whenever
\begin{equation}
\frac{1}{\Delta_{n}m_{n}}\sum_{j=1}^{m_{n}-1}t_{j}^{2}\mathbb{E}\left(\left|\xi_{j,n}\right|^{2}\right)=\frac{\Delta_{n}}{m_{n}}\mathrm{Var}\left(S_{m_{n}}^{\Delta_{n},1}\right)\rightarrow\sigma_{a}^{2},\,\,\,n\rightarrow\infty,\label{condvarianceCLT}
\end{equation}
and if the Lyapunov condition is satisfied, i.e. for some $p>2$

\begin{equation}
I_{n,p,1}:=\sum_{j=1}^{m_{n}-1}t_{j}^{p}\mathbb{E}\left(\left|\xi_{j,n}\right|^{p}\right)=\mathrm{o}\left((m_{n}\Delta_{n})^{p/2}\right),\,\,\,\text{as }n\rightarrow\infty.\label{eq:lyapunov}
\end{equation}
Lemmas \ref{lemmavarianceSm} and \ref{leadingtermshortmemory} immediately
imply (\ref{condvarianceCLT}). Now fix $p_{0}\land3>p>2$. Thanks
to (\ref{lebpartition}), we deduce that the L\'{e}vy measure of $\xi_{j,n}$
is given by
\[
\nu_{\xi_{j,n}}(\cdot)=(m_{n}-j)\int_{t_{j-1}}^{t_{j}}\left[a(s)-a(s+\Delta_{n})\right]\mathrm{d}s\nu(\cdot),\,\,j=1,\ldots,m_{n}-1.
\]
Therefore, from Corollary 1.2.7. in \cite{Turner11}, there is a constant
$C>0$ only depending on $p$ and $\nu(\cdot)$, such that 
\[
\mathbb{E}\left(\left|\xi_{j,n}\right|^{p}\right)\leq C\max\left\{ (m_{n}-j)\int_{t_{j-1}}^{t_{j}}\left[a(s)-a(s+\Delta_{n})\right]\mathrm{d}s,\left((m_{n}-j)\int_{t_{j-1}}^{t_{j}}\left[a(s)-a(s+\Delta_{n})\right]\mathrm{d}s\right)^{p/2}\right\} .
\]
Hence
\begin{equation}
I_{n,p,1}\leq C(I_{n,p,1}^{(1)}+I_{n,p,1}^{(2)}),\label{boundInp1}
\end{equation}
where we have let 
\begin{align*}
I_{n,p,1}^{(1)} & :=\sum_{j=1}^{m_{n}-1}(t_{m_{n}}-t_{j})\int_{t_{j-1}}^{t_{j+1}}\left|t_{j}\right|^{p}\mathrm{d}\left|a\right|(s);\\
I_{n,p,1}^{(2)} & :=\sum_{j=1}^{m_{n}}(t_{m_{n}}-t_{j})^{p/2}\left(\int_{t_{j-1}}^{t_{j+1}}\left|t_{j}\right|^{2}\mathrm{d}\left|a\right|(s)\right)^{p/2}.
\end{align*}
Thus, (\ref{eq:lyapunov}) is obtained whenever $I_{n,p,1}^{(1)}+I_{n,p,1}^{(2)}=\mathrm{o}\left((m_{n}\Delta_{n})^{p/2}\right)$.
Observe that for any $j=1,\ldots,m_{n}-1$, $t_{j-1}\leq\zeta\leq t_{j}$
and $p_{0}>q\geq2$, it holds that 
\[
\int_{t_{j-1}}^{t_{j}}\left|\zeta^{q}-s^{q}\right|\mathrm{d}\left|a\right|(s)\leq C(\Delta_{n}\mu_{q-1,a}(t_{j-1},t_{j}]+\Delta_{n}^{2}\mu_{q-2,a}(t_{j-1},t_{j}]+\Delta_{n}^{q}\mu_{0,a}(t_{j-1},t_{j}]).
\]
Furthermore, since the distribution function
\[
F_{p,a}(x):=\mu_{p,a}(-\infty,x],\,\,\,x\in\mathbb{R},p<p_{0},
\]
is continuous and bounded, it is also uniformly continuous on $\mathbb{R}$.
Hence, as $n\rightarrow\infty$
\[
\max_{1,\ldots,n}\mu_{p,a}(t_{j-1},t_{j}]=\max_{1,\ldots,n}\left|F_{p,a}(t_{j})-F_{p,a}(t_{j-1})\right|\rightarrow0,\,\,\,\,2\leq p<p_{0}.
\]
Using the previous properties and the fact that $\mu_{p,a}(\mathbb{R})<\infty$,
one easily deduce that for any $p_{0}>p>2$ 
\begin{align}
\frac{1}{(m_{n}\Delta_{n})^{p/2}}\left|I_{n,p,1}^{(1)}\right| & \leq C\frac{1}{(m_{n}\Delta_{n})^{p/2-1}}\sum_{j=1}^{m_{n}-1}\int_{t_{j-1}}^{t_{j}}\left|t_{j}\right|^{p}\mathrm{d}\left|a\right|(s)\label{ineqInp1(1)}\\
 & \leq C\frac{1}{(m_{n}\Delta_{n})^{p/2-1}}+\mathrm{o}(1)\rightarrow0,\nonumber 
\end{align}
and that
\begin{align*}
\frac{1}{(m_{n}\Delta_{n})^{p/2}}\left|I_{n,p,1}^{(2)}\right| & \leq C\sum_{j=1}^{m_{n}-1}\left\{ \int_{t_{j-1}}^{t_{j}}\left|t_{j}\right|^{2}\mathrm{d}\left|a\right|(s)\right\} ^{p/2}\\
 & \leq C\max_{j=1,\ldots,m_{n}}\mu_{l,a}(t_{j-1},t_{j}]^{p/2-1}+\mathrm{o}(1)\rightarrow0,
\end{align*}
which concludes the argument for (\ref{cltshortmemory}). 

Now, let $\lambda_{0},\ldots,\lambda_{r}\in\mathbb{R}$ and $0=t_{0}<t_{1}<\cdots<t_{r}=1$.
To show the convergence of the finite-dimensional distributions, we
are going to verify that 
\begin{equation}
\sqrt{\frac{\Delta_{n}}{n}}\sum_{q=1}^{r}\lambda_{q}(\mathbf{S}_{t_{q}}^{n}-\mathbf{S}_{t_{q-1}}^{n})\overset{d}{\rightarrow}\sigma_{a}\left(\sum_{q=1}^{r}\lambda_{q}^{2}(t_{q}-t_{q-1})\right)^{1/2}N(0,1).\label{ffdconv}
\end{equation}
Thanks to Lemma (\ref{leadingtermshortmemory})
\[
\sqrt{\frac{\Delta_{n}}{n}}\sum_{q=1}^{r}\lambda_{q}(\mathbf{S}_{t_{q}}^{n}-\mathbf{S}_{t_{q-1}}^{n})=\sqrt{\frac{\Delta_{n}}{n}}\sum_{q=1}^{r}\lambda_{q}\left(S_{\left[nt_{q}\right]}^{\Delta_{n},1}-S_{\left[nt_{q-1}\right]}^{\Delta_{n},1}\right)+\mathrm{o}_{\mathbb{P}}(1),\,\,\,n\in\mathbb{N},
\]
Moreover, in view that 
\[
\sum_{q=1}^{r}\lambda_{q}\left(S_{\left[nt_{q}\right]}^{\Delta_{n},1}-S_{\left[nt_{q-1}\right]}^{\Delta_{n},1}\right)=\sum_{i=1}^{n-1}\sum_{j=i}^{n-1}d_{n,j,i}\chi_{i,j}^{\Delta},
\]
where
\[
d_{n,j,i}:=\sum_{q=1}^{r}\lambda_{q}\sum_{k=i}^{j}\mathbf{1}_{\{\left[nt_{q-1}\right]\leq k\leq\left[nt_{q}\right]-1\}},
\]
it follows that for $p_{0}\land3>p>2$, 
\begin{equation}
\left(\frac{\Delta_{n}}{n}\right)^{p/2}\mathbb{E}\left(\sum_{i=0}^{n-1}\sum_{j=i}^{n-1}\left|d_{n,j,i}\chi_{i,j}^{\Delta}\right|^{p}\right)\leq C\frac{1}{(m_{n}\Delta_{n})^{p/2}}I_{n,p,1}\rightarrow0.\label{lyapunovfdd}
\end{equation}
where we have used (\ref{eq:lyapunov}). Hence, in view that $L$
is independently scattered, (\ref{ffdconv}) is obtained whenever
\begin{equation}
\frac{\Delta_{n}}{n}\mathbb{E}\left\{ \left[\sum_{q=1}^{r}\lambda_{q}\left(S_{\left[nt_{q}\right]}^{\Delta_{n},1}-S_{\left[nt_{q-1}\right]}^{\Delta_{n},1}\right)\right]^{2}\right\} \rightarrow\sigma_{a}^{2}\sum_{q=1}^{r}\lambda_{q}^{2}(t_{q}-t_{q-1})^{.},\,\,\,\text{as }n\rightarrow\infty.\label{varianceconvshortmem}
\end{equation}
Since for any $N>M>K>U$
\begin{align}
S_{N}^{\Delta,1}-S_{M}^{\Delta,1} & =\sum_{i=1}^{M-1}\sum_{j=M}^{N-1}(j-i+1)\chi_{i,j}^{\Delta}+\sum_{i=M}^{N-1}\sum_{j=i}^{N-1}(j-i+1)\chi_{i,j}^{\Delta},\label{IncremS11}\\
S_{K}^{\Delta,1}-S_{J}^{\Delta,1} & =\sum_{i=1}^{U-1}\sum_{j=U}^{K-1}(j-i+1)\chi_{i,j}^{\Delta}+\sum_{i=U}^{K-1}\sum_{j=i}^{K-1}(j-i+1)\chi_{i,j}^{\Delta},\label{IncremS12}
\end{align}
it follows that 
\[
\mathbb{E}\left[\left(S_{\left[nt_{q}\right]}^{\Delta_{n},1}-S_{\left[nt_{q-1}\right]}^{\Delta_{n},1}\right)\left(S_{\left[nt_{l}\right]}^{\Delta_{n},1}-S_{\left[nt_{l-1}\right]}^{\Delta_{n},1}\right)\right]=0,\,\,\,\text{if }l\neq q.
\]
Combining Lemmas \ref{lemmavarianceSm} and \ref{leadingtermshortmemory},
and the stationarity of $X$, (\ref{varianceconvshortmem}) follows
immediately.\end{proof}

\begin{proof}[Step 2: Tightness]Using similar arguments as in the
proof of Lemma 2.1 in \cite{Taqqu75} and Lemma \ref{lemmavarianceSm},
we deduce that $\sqrt{\frac{\Delta_{n}}{n}}\mathbf{S}^{n}$ is tight
if for any sequence $m_{n}\in\mathbb{N}$, such that $m_{n}\uparrow\infty$,
$\Delta_{n}m_{n}\rightarrow\infty$ and $\Delta_{n}\downarrow0$,
as $n\rightarrow\infty$, it holds that
\begin{equation}
\mathbb{E}\left(\left|S_{m_{n}}^{\Delta_{n}}\right|^{p}\right)=\mathrm{O}\left(\left(\frac{m_{n}}{\Delta_{n}}\right)^{p/2}\right),\label{tightnesscond}
\end{equation}
for some $p>2$. We proceed now to show that (\ref{tightnesscond})
is fulfilled . By (\ref{eq:decompSn}), we have that
\[
\left(\frac{\Delta_{n}}{m_{n}}\right)^{p/2}\mathbb{E}\left(\left|S_{m_{n}}^{\Delta_{n}}\right|^{p}\right)\leq C\left\{ I_{n,p,1}+I_{n,p,2}+I_{n,p,3}+I_{n,p,4}\right\},
\]
where $I_{n,p,1}$ is as in step 1, and
\begin{align*}
I_{n,p,2} & =\frac{1}{(m_{n}\Delta_{n})^{p/2}}\sum_{i=1}^{m_{n}-1}(t_{m_{n}}-t_{i})^{p}\mathbb{E}\left(\left|\chi_{i,m_{n}-1}^{\Delta_{n}}\right|^{p}\right);\\
I_{n,p,3} & =\frac{1}{(m_{n}\Delta_{n})^{p/2}}\sum_{i=1}^{m_{n}-1}(t_{j+1})^{p}\mathbb{E}\left(\left|\chi_{0,j}^{\Delta_{n}}\right|^{p}\right);\\
I_{n,p,4} & =\left(\Delta_{n}m_{n}\right)^{p/2}\mathbb{E}\left(\left|\chi_{0,m_{n}-1}^{\Delta_{n}}\right|^{p}\right).
\end{align*}
We have already seen that $I_{n,p,1}\rightarrow0$ as $n\rightarrow\infty$.
On the other hand, invoking once again Corollary 1.2.7. in \cite{Turner11}
and using (\ref{lebpartition}), we get
\begin{align*}
\mathbb{E}\left(\left|\chi_{i,m_{n}-1}^{\Delta_{n}}\right|^{p}\right) & \leq C\max\left\{ \int_{t_{i-1}}^{t_{i}}a(t_{m_{n}-1}-s)\mathrm{d}s,\left(\int_{t_{i-1}}^{t_{i}}a(t_{m_{n}-1}-s)\mathrm{d}s\right)^{p/2}\right\} ,\\
\mathbb{E}\left(\left|\chi_{0,j}^{\Delta_{n}}\right|^{p}\right) & \leq C\max\left\{ \int_{t_{j-1}}^{t_{j}}a(s)\mathrm{d}s,\left(\int_{t_{i-1}}^{t_{i}}a(s)\mathrm{d}s\right)^{p/2}\right\} ,\\
\mathbb{E}\left(\left|\chi_{0,m_{n}-1}^{\Delta_{n}}\right|^{p}\right) & \leq C\max\left\{ \int_{t_{m_{n}-1}}^{\infty}a(s)\mathrm{d}s,\left(\int_{t_{m_{n}-1}}^{\infty}a(s)\mathrm{d}s\right)^{p/2}\right\} .
\end{align*}
Similarly as in Step 1, we deduce that for $p_{0}\land3>p>2$ and
$n$ large enough, the following estimates are valid
\begin{align*}
I_{n,p,2} & \leq C\frac{1}{(m_{n}\Delta_{n})^{p/2}}\sum_{i=1}^{m_{n}-1}(t_{m_{n}}-t_{i})^{p}\int_{t_{i-1}}^{t_{i}}a(t_{m-1}-s)\mathrm{d}s,\\
I_{n,p,3} & \leq C\frac{1}{(m_{n}\Delta_{n})^{p/2}}\sum_{i=1}^{m_{n}-1}(t_{j+1})^{p}\int_{t_{j-1}}^{t_{j}}a(s)\mathrm{d}s,\\
I_{n,p,4} & \leq C\left(\Delta_{n}m_{n}\right)^{p/2}\int_{\Delta_{n}\left(m_{n}-1\right)}^{\infty}a(s)\mathrm{d}s.
\end{align*}
Assumption \ref{assumptionshortmem} now asserts that as $n\rightarrow\infty$
\[
I_{p,4}\leq C(\Delta_{n}m_{n})^{p/2-p_{0}+1}\rightarrow0,
\]
because $2<p<p_{0}<2(p_{0}-1)$. Furthermore, analogous arguments
used to establish (\ref{approxvarS2}) and (\ref{approxvarS3}), can
be applied in order to get that
\begin{align*}
I_{p,2}+I_{p,3}= & C\frac{1}{(\Delta_{n}m_{n})^{p/2}}\int_{0}^{\Delta_{n}m_{n}}s^{p}a(s)\mathrm{d}s+\mathrm{o}(1)\rightarrow0.
\end{align*}
All in all give us
\[
\left(\frac{\Delta_{n}}{m_{n}}\right)^{p/2}\mathbb{E}\left(\left|S_{m_{n}}^{\Delta_{n}}\right|^{p}\right)\rightarrow0,
\]
being this enough for (\ref{tightnesscond}).

\end{proof}

\begin{proof}[Step 3: Stability]From Step 2, Proposition 3.9 in \cite{HauslerLuschgy15}
and its subsequent remark, the stable convergence in $\mathcal{D}([0,1])$
will be obtained if (\ref{varianceconvshortmem}) can be strengthened
to $\mathcal{G}^{X}$-stable convergence. Consider the filtration
\begin{equation}
\mathcal{F}_{i}^{n}:=\sigma\left(\left\{ L(\mathcal{\mathcal{P}}_{A}^{\Delta_{n}}(k,j))\right\} _{j\geq k}:k=0,\ldots,i\right),\,\,\,i=0,\ldots,n-1,\,\,n\in\mathbb{N}.\label{eq:filtration}
\end{equation}
From Step 1
\[
\sqrt{\frac{\Delta_{n}}{n}}\sum_{q=1}^{r}\lambda_{q}(\mathbf{S}_{t_{q}}^{n}-\mathbf{S}_{t_{q-1}}^{n})=\frac{1}{\sqrt{n\Delta_{n}}}\sum_{i=1}^{n-1}\tilde{\xi}_{i,n}+\mathrm{o}_{\mathbb{P}}(1),\,\,\,n\in\mathbb{N},
\]
where the array
\[
\tilde{\xi}_{i,n}=\sum_{j=i}^{n-1}d_{n,j,i}\chi_{i,j}^{\Delta}
\]
is $\mathcal{F}_{i}^{n}$-measurable and independent of $\mathcal{F}_{i-1}^{n}$
for all $i=1,\ldots,n-1$. Consequently, thanks to (\ref{varianceconvshortmem}),
(\ref{lyapunovfdd}) and Theorem 6.1 in \cite{HauslerLuschgy15},
it holds that 
\[
\sum_{q=1}^{r}\lambda_{q}(\mathbf{S}_{t_{q}}^{n}-\mathbf{S}_{t_{q-1}}^{n})\overset{\mathcal{G}-d}{\rightarrow}\sigma_{a}\left(\sum_{q=1}^{r}\lambda_{q}^{2}(t_{q}-t_{q-1})\right)^{1/2}N(0,1),
\]
where
\[
\mathcal{G}:=\sigma(\cup_{n\geq1}\cap_{N\geq n}\mathcal{F}_{n}^{N}),
\]
but in view that 
\[
\sigma(X_{0},X_{\Delta_{n}},\ldots,X_{i\Delta_{n}})\subseteq\mathcal{F}_{i}^{n},\,\,\,i=0,1,\ldots,m_{n}-1,
\]
it follows immediately that $\mathcal{G}^{X}\subseteq\mathcal{G},$
which concludes the proof.

\end{proof}

\subsection{Proof of Theorems \ref{theoremmean-2}-\ref{theoremmean-2-1}}

In this subsection we justify the statements of Theorems \ref{theoremmean}-\ref{theoremmean3}.
In what follows $(\gamma,b,\nu)$ and $\psi$ will denote respectively,
the characteristic triplet and exponent of $L$. We note that, for
the sake of exposition, the proof of each theorem is given in a corresponding
subsubsection.

\subsubsection{$b>0$}

For every $n\in\mathbb{N}$ we will let $r_{n}=\frac{1}{n\sqrt{a(n\Delta_{n})n\Delta_{n}}}$.
Observe that the same argument used in step 2 in the proof Theorem
\ref{theoremmean} together with Lemma \ref{lemmavarianceSm}, give
us automatically that $r_{n}\mathbf{S}^{n}$ is tight in $\mathcal{D}([0,1])$
if $\mathbb{E}\left(\left|L^{\prime}\right|^{2}\right)<\infty$. Therefore,
we only need to show that within the framework of Assumption \ref{assumptionlongmem},
the finite-dimensional distributions of $r_{n}\mathbf{S}^{n}$ converge
stably to those of the fBm with index $H=2-\kappa$/2. 

Before proceeding with the proof we would like to emphasize that in
this situation the Lyapunov condition is not satisfied in general.
Indeed, for instance if $L$ is symmetric, we have that 
\[
r_{n}^{2}\mathrm{Var}(S_{n}^{\Delta_{n},4})=\frac{n\Delta_{n}}{a(n\Delta_{n})(n\Delta_{n})^{2}}\int_{m_{n}\Delta_{n}}^{\infty}a(s)\mathrm{d}s\rightarrow\frac{1}{\kappa-1},
\]
but from Corollary 1.2.7. in \cite{Turner11}, for $n$ large enough
\[
C\left(\int_{n\Delta_{n}}^{\infty}a(s)\mathrm{d}s\right)^{p/2}\leq\mathbb{E}\left(\left|\chi_{0,n-1}^{\Delta_{n}}\right|^{p}\right).
\]
Thus, by Rosenthal's inequality, it holds that for any $p>2$
\[
r_{n}^{p}\mathbb{E}\left(\left|S_{n}^{\Delta_{n}}\right|^{p}\right)\geq C\left(\frac{1}{a(n\Delta_{n})n\Delta_{n}}\int_{n\Delta_{n}}^{\infty}a(s)\mathrm{d}s\right)^{p/2}\rightarrow\left(\frac{1}{\kappa-1}\right)^{p/2}>0.
\]
Nevertheless, from the proof of Theorem \ref{theoremmean-2} \textbf{ii.}
below, and the Lévy-Itô decomposition for L\'{e}vy bases, it follows that
the non-Gaussian component of $r_{n}\mathbf{S}^{n}$ is negligible.
Consequently, we may and do assume in this part that $\gamma=0$ and
$\nu\equiv0$.

\begin{proof}[Proof of Theorem \ref{theoremmean-2}]

Fix $q\in\mathbb{N}$, $\lambda_{1},\ldots,\lambda_{q}\in\mathbb{R}$
and $0=u_{0}<u_{1}<\cdots<u_{q}=1$. Let 
\[
\tilde{\xi}_{i,n}:=\sum_{j=i}^{n-1}\theta_{j,i}\chi_{i,j}^{\Delta_{n}}+\theta_{n-1,i}\zeta{}_{i,n}^{\Delta_{n}},
\]
where $\zeta{}_{i,n}^{\Delta}=\sum_{j\geq n}\chi_{i,j}^{\Delta_{n}}$
and
\[
\theta_{j,i}=\sum_{k=i}^{j}\sum_{l=1}^{q}\sum_{m=l}^{q}\lambda_{m}\mathbf{1}_{\left[nu_{l-1}\right]\leq k<\left[nu_{l}\right]}.
\]
Observe that $\tilde{\xi}_{i,n}$ is $\mathcal{F}_{i}^{n}$-measurable
(see (\ref{eq:filtration})), independent of $\mathcal{F}_{i-1}^{n}$,
centered, Gaussian, and such that
\begin{align*}
r_{n}\sum_{l=1}^{q}\lambda_{l}\mathbf{S}_{u_{l}}^{n} & =r_{n}\sum_{i=0}^{n-1}\tilde{\xi}_{i,n},\,\,\,n\in\mathbb{N}.
\end{align*}
Thus, thanks to the \mikko{Lindeberg}-Feller Theorem and Theorem 6.1 in \cite{HauslerLuschgy15},
for the stable convergence of the f.d.d. of $r_{n}\mathbf{S}^{n}$
to those of the fBm, we only need to check that as $n\rightarrow\infty$
\begin{equation}
r_{n}^{2}\mathbb{E}(\mathbf{S}_{t}^{n}\mathbf{S}_{u}^{n})\rightarrow\sigma_{\kappa}^{2}\left\{ t^{2H}+u^{2H}+\left|t-u\right|^{2H}\right\} ,\,\,\,\forall\,t,u\in[0,1].\label{covfb}
\end{equation}
Let $1\geq t\geq u\geq0$. In view that 
\[
\mathbb{E}(\mathbf{S}_{t}^{n}\mathbf{S}_{u}^{n})=\frac{1}{2}\left\{ \sum_{i=0}^{\left[nu\right]-1}\sum_{j=0}^{\left[nu\right]-1}\varGamma_{X}\left(\left|t_{i}-t_{j}\right|\right)+\sum_{i=0}^{\left[nt\right]-1}\sum_{j=0}^{\left[nt\right]-1}\varGamma_{X}\left(\left|t_{i}-t_{j}\right|\right)-\sum_{i=\left[nu\right]}^{\left[nt\right]-1}\sum_{j=\left[nu\right]}^{\left[nt\right]-1}\varGamma_{X}\left(\left|t_{i}-t_{j}\right|\right)\right\} ,
\]
we can use analogous estimates derived in the proof of Lemma \ref{lemmavarianceSm}
to deduce that
\begin{align*}
\sum_{i=0}^{\left[nt\right]-1}\sum_{j=0}^{\left[ns\right]-1}\varGamma_{X}\left(\left|t_{i}-t_{j}\right|\right) & =\frac{1}{\Delta_{n}^{2}}\int_{0}^{\left[nt\right]\Delta_{n}}\int_{0}^{r}\varGamma_{X}\left(s\right)\mathrm{d}s\mathrm{d}r+\frac{1}{\Delta_{n}^{2}}\int_{0}^{\left[ns\right]\Delta_{n}}\int_{0}^{r}\varGamma_{X}\left(s\right)\mathrm{d}s\mathrm{d}r\\
 & -\frac{1}{\Delta_{n}^{2}}\int_{0}^{\left[nt\right]\Delta_{n}-\left[ns\right]\Delta_{n}}\int_{0}^{r}\varGamma_{X}\left(s\right)\mathrm{d}s\mathrm{d}r+\mathrm{O}(n).
\end{align*}
Relation (\ref{covfb}) now follows easily from this and KT.

\end{proof}

\subsubsection{$b=0$}

In this part, unless otherwise said, we will always assume that $b=0$.
Observe that under the assumptions of Theorem \ref{theoremmean3} ii.
(\ref{chfctat0}) is once again valid. We recall to thew reader that
we are also assuming that $L$ has mean zero. Finally, we would like
to stress that by replacing $\beta_{\nu}$ by $2$ below, it follows
that
\[
\frac{1}{n\sqrt{a(n\Delta_{n})n\Delta_{n}}}\mathbf{S}_{t}^{n}\overset{\mathbb{P}}{\rightarrow}0,\,\,\,n\rightarrow\infty.
\]
 We now proceed to present a proof for Theorem \ref{theoremmean3}.

\begin{proof}[Proof of Theorem \ref{theoremmean3}]The proof is organized
as follow: First, based on our assumption, we derive some preliminary
estimates. Secondly, by using (\ref{eq:decompSn}), we approximate
the characteristic function of $\mathbf{S}_{t}^{n}$ by means of $I_{n}^{\Delta,1}(\cdot)$
and $I_{n}^{\Delta,2}(\cdot)$, where the latter are as in Lemma \ref{lemmaapproxchf}.
We conclude by applying such approximation to obtain the desired result.
For the rest of the proof we will use the notation $T_{n}=n\Delta_{n}$,
$r_{n}=a(T_{n})T_{n}$, and$\alpha=\kappa-1$ as well as 
\[
B_{n}=\begin{cases}
\Delta_{n}/(c_{a}T_{n})^{\frac{1}{\kappa-1}} & \text{if }\beta_{\nu}<\kappa-1;\\
1/nr_{n}^{1/\beta_{\nu}} & \text{if }2>\beta_{\nu}>\kappa-1.
\end{cases}
\]

\textbf{Preliminary estimates:} First, by the mean of Assumption \ref{assumptionlongmem-1},
we can invoke the so-called Potter's bounds (see Theorem 1.5.6 in
\cite{BinghamGoldTeu87}). Such a result provides the existence of
a positive constant only depending on $\varepsilon>0$, such that
for all $0\leq r\leq1$
\begin{equation}
\frac{a(T_{n}s)}{a(T_{n})}\leq C\left(s^{-\alpha-\varepsilon}\lor s^{-\alpha+\varepsilon}\right),\,\,\text{and}\,\,\,\frac{a^{\prime}(r\Delta_{n}+T_{n}s)}{a^{\prime}(T_{n})}\leq C\left(s^{-\kappa-\varepsilon}\lor s^{-\kappa+\varepsilon}\right),\,\,\,s>0.\label{potterboundsfortrawlfunct}
\end{equation}
From Lemma \ref{lemmaaproxZ}, for every $2\geq\theta>\beta_{\nu}$,
we have that for $\left|z\right|>1$
\begin{align}
\left|\psi\left(z\right)\right| & \leq C\left(\left|z\right|^{2}\land\left|z\right|^{\theta}\right),\label{bounchfnctsmallgethoor}
\end{align}
while for $\left|z\right|\leq1$, by the square integrability of $L^{\prime}$
we easily obtain that
\[
\left|\psi\left(z\right)\right|\leq C\left|z\right|^{2}.
\]
On the other, using that $\mathbb{E}\left(\left|L^{\prime}\right|^{2}\right)<\infty$,
it follows that as $x\rightarrow\infty$, $\nu^{\pm}(x)=\mathrm{O}(x^{\theta})$,
for all $\theta\leq2$. Thus, if $\nu^{\pm}(x)\sim\tilde{K}_{\pm}x^{-\beta_{\nu}}$
as $x\rightarrow0^{+}$, for some constants $\tilde{K}_{+}+\tilde{K}_{-}>0$,
then the càdlàg function $\ell(x):=\bar{v}(x)x^{\beta_{\nu}}:=\left[v^{+}(x)+v^{-}(x)\right]x^{\beta_{\nu}}$
is uniformly bounded in $[0,\infty)$. Consequently, by Lemma \ref{lemmaaproxZ},
for any $y>0$, $z\in\mathbb{R}$
\begin{equation}
\left|\psi\left(zy\right)\right|\leq C(\left|z\right|\lor1)^{2}\int_{0}^{1}\bar{v}(x/y)x\mathrm{d}x\leq Cy^{\beta_{\nu}}.\label{boundchfncgethoorbig}
\end{equation}

\textbf{Estimating the characteristic function:} Recall that from
(\ref{eq:decompSn}), the following decomposition holds
\[
\mathcal{C}(z\ddagger S_{n}^{\Delta_{n}})=\sum_{k=1}^{4}\mathcal{C}(z\ddagger S_{n}^{\Delta_{n},k}),\,\,\,n\in\mathbb{N},z\in\mathbb{R}.
\]
From Lemma \ref{lemmaapproxchf} and Assumption \ref{assumptionlongmem-1},
and put for $l=2,3$
\begin{align*}
\mathcal{E}_{1}^{n}(z):=\left|\mathcal{C}\left(z\ddagger S_{n}^{\Delta_{n},1}\right)-I_{n}^{\Delta,1}\left(z\right)\right|\leq & C\left|z\right|^{2}\left(\frac{T_{n}+1}{\Delta_{n}}\right)\\
 & +C\Delta_{n}\int_{0}^{1}\int_{0}^{T_{n}}\left|\psi\left(\frac{sz}{\Delta_{n}}\right)\right|a^{\prime}(r\Delta_{n}+s)\mathrm{d}s\mathrm{d}r;\\
\mathcal{E}_{l}^{n}(z):=\left|\mathcal{C}\left(z\ddagger S_{n}^{\Delta_{n},l}\right)-I_{n}^{\Delta,2}(z)\right|\leq & C\frac{\left|z\right|^{2}}{\Delta_{n}}\int_{0}^{T_{n}}a(s)s\mathrm{d}s.
\end{align*}
Under the conditions stated in \textbf{ii.,} \textbf{(}\ref{boundchfncgethoorbig}\textbf{)
and (}\ref{potterboundsfortrawlfunct}\textbf{) imply that}
\begin{align*}
\mathcal{E}_{1}^{n}\left(B_{n}z\right) & \leq C_{z}\left(\frac{1}{r_{n}^{2/\beta_{\nu}}n}+\frac{1}{nT_{n}r_{n}^{2/\beta_{\nu}}}\right)\\
 & +C\frac{T_{n}a^{\prime}(T_{n})}{na(T_{n})}\int_{0}^{1}\int_{0}^{1}s^{\beta_{\nu}-\kappa-\varepsilon}\mathrm{d}s\mathrm{d}r;\\
\mathcal{E}_{l}^{n}\left(B_{n}z\right) & \leq C\frac{1}{nr_{n}^{2/\beta_{\nu}}T_{n}}\int_{0}^{T_{n}}sa(s)\mathrm{d}s.
\end{align*}
Given that $(1-\alpha)(2-\beta_{\nu})+\beta_{\nu}>0$, it is easy
to see that $r_{n}^{(2-\beta_{\nu})/\beta_{\nu}}T_{n}\rightarrow\infty$
as $n\rightarrow\infty$. KT then now asserts that $\mathcal{E}_{l}^{n}\left(B_{n}z\right)\rightarrow0$
for $l=2,3$. Furthermore, since $\beta_{\nu}>\kappa-1$, we have
that $2(1-\alpha)+\beta_{\nu}>0$, meaning this that $r_{n}^{2/\beta_{\nu}}T_{n}\rightarrow\infty$,
in other words $\mathcal{E}_{1}^{n}\left(B_{n}z\right)$ is negligible
when $n$ tends to $+\infty$. On the other hand, in the situation
in which $\beta_{\nu}<\kappa-1$, KT and (\ref{bounchfnctsmallgethoor})
imply that
\begin{align*}
\mathcal{E}_{1}^{n}\left[B_{n}z\right] & \leq C\frac{\Delta_{n}}{T_{n}^{\frac{3-\kappa}{\kappa-1}}}+C\Delta_{n}\int_{0}^{1}\int_{0}^{T_{n}}\left(\left|s/T_{n}^{\frac{1}{\kappa-1}}\right|^{2}\land\left|s/T_{n}^{\frac{1}{\kappa-1}}\right|^{\theta}\right)a^{\prime}(r\Delta_{n}+s)\mathrm{d}s\mathrm{d}r;\\
\mathcal{E}_{l}^{n}\left[B_{n}z\right] & \leq C\frac{\Delta_{n}a(T_{n})T_{n}^{2}}{T_{n}^{\frac{2}{\kappa-1}}}\mathrm{O}(1).
\end{align*}
Using that $(3-\kappa)(\kappa-1)<2$ and KT, results in $\mathcal{E}_{l}^{n}\left[B_{n}z\right]$
vanishing as $n\rightarrow\infty.$ Moreover, the change of variable
$x=s/T_{n}^{\frac{1}{\kappa-1}}$ and using (\ref{potterboundsfortrawlfunct})
we deduce that for $\kappa-1>\theta>\beta_{\nu}$ and $0<\varepsilon<(3-\kappa)\land(\kappa-1-\theta)$
\begin{equation}
\mathcal{E}_{1}^{n}\left[B_{n}z\right]\leq C\left(\int_{0}^{\infty}\left(\left|s\right|^{2}\land\left|s\right|^{\theta}\right)s^{-\kappa-\varepsilon}\mathrm{d}s\right)a^{\prime}(T_{n}^{\frac{1}{\kappa-1}})T_{n}^{\frac{1}{\kappa-1}}\rightarrow0,\,\,\,n\rightarrow\infty,\label{bound1}
\end{equation}
we we have used Assumption \ref{assumptionlongmem-1}. 

\textbf{Convergence in i.:} So far we have shown that 

\begin{align*}
\mathcal{C}(z\ddagger B_{n}S_{n}^{\Delta_{n},1}) & =I_{n}^{\Delta,1}\left(B_{n}z\right)+\mathrm{o}(1);\\
\mathcal{C}(z\ddagger B_{n}S_{n}^{\Delta_{n},l}) & =I_{n}^{\Delta,2}\left(B_{n}z\right)+\mathrm{o}(1).
\end{align*}

We now proceed to show that the convergence in \textbf{i.} holds.
We start by verifying that for $l=2,3,4,$ $S_{n}^{\Delta_{n},l}$
are negligible when $\beta_{\nu}<\kappa-1$. Indeed, from (\ref{bounchfnctsmallgethoor})
and Assumption \ref{assumptionlongmem-1}, for all $z\neq0$, $\kappa-1>\theta>\beta_{\nu}$,
and $n$ large enough
\[
\left|\mathcal{C}\left(z\ddagger B_{n}S_{n}^{\Delta_{n},4}\right)\right|=\left|\psi\left(zT^{\frac{\kappa-2}{\kappa-1}}\right)\int_{T_{n}}^{\infty}a(s)\mathrm{d}s\right|\leq CT^{\theta\frac{\kappa-2}{\kappa-1}}T_{n}^{2-\kappa}\rightarrow0.
\]
Reasoning as in (\ref{bound1}), we deduce from (\ref{potterboundsfortrawlfunct})
that for $n$ enough large and $4-\kappa>\varepsilon>0$
\begin{align}
\left|I_{n}^{\Delta,2}\left(B_{n}z\right)\right| & \leq CT_{n}^{\frac{1}{\kappa-1}}\int_{0}^{T_{n}^{\frac{\kappa-2}{\kappa-1}}}\left(\left|s\right|^{2}\land\left|s\right|^{\theta}\right)a(T_{n}^{\frac{1}{\kappa-1}}s)\mathrm{d}s\label{negliI2smallBG}\\
 & \leq\mathrm{o}(1)+T_{n}^{-1}\int_{1}^{T_{n}^{\frac{\kappa-2}{\kappa-1}}}\left|s\right|^{\theta-\kappa+1+\varepsilon}\mathrm{d}s\nonumber \\
 & \leq\mathrm{o}(1)+CT_{n}^{\frac{(\kappa-2)(\theta-\kappa+2+\varepsilon)-(\kappa-1)}{\kappa-1}}.\nonumber 
\end{align}
Since $\theta-\kappa+2+\varepsilon<1+\varepsilon$, then by choosing
$\varepsilon<1/(\kappa-2)$, we obtain from above that $I_{n}^{\Delta,2}\left(B_{n}z\right)\rightarrow0$.
Using the previous reasoning, we conclude that in this case 
\[
B_{n}\mathbf{S}_{t}^{n}=B_{n}S_{\left[nt\right]}^{\Delta_{n}}+\mathrm{o}_{\mathbb{P}}(1).
\]
Furthermore, in view of (\ref{IncremS11}) and (\ref{IncremS12}),
$(S_{\left[nt\right]}^{\Delta_{n}})_{t\geq0}$ has independent increments.
From this observation and the stationarity of $X$, in order to finish
the proof, we only need to check that as $n\rightarrow\infty$
\[
\mathcal{C}\left(z\ddagger B_{n}S_{\left[nt\right]}^{\Delta_{n}}\right)\rightarrow\mathcal{C}(z\ddagger Y_{t}),
\]
where $Y$ as in the theorem. For simplicity we let $t=1$. Suppose
now that $z>0$, then by doing the change of variables $x=(sz)/(c_{a}T_{n})^{\frac{1}{\kappa-1}}$
we see that
\[I_{n}^{\Delta,1}\left(B_{n}z\right)\]
\[
=z^{-1}c_{a}^{^{\frac{1}{\kappa-1}}}a^{\prime}(T_{n}^{\frac{1}{\kappa-1}})T_{n}^{\frac{\kappa}{\kappa-1}}\int_{0}^{1}\int_{0}^{T_{n}^{\frac{\kappa-2}{\kappa-1}}zc_{a}^{1/(\kappa-1)}}\left[1-x\left(T_{n}^{\frac{\kappa-2}{\kappa-1}}zc_{a}^{1/(\kappa-1)}\right)^{-1}\right]\psi(x)\frac{a^{\prime}\left[r\Delta_{n}+\frac{(c_{a}T_{n})^{\frac{1}{\kappa-1}}}{z}x\right]}{a^{\prime}(T_{n}^{\frac{1}{\kappa-1}})}\mathrm{d}x\mathrm{d}r.
\]
Using a similar argument as in (\ref{negliI2smallBG}) shows that
\[
\left|\left[1-x\left(T_{n}^{\frac{\kappa-2}{\kappa-1}}zc_{a}^{1/(\kappa-1)}\right)^{-1}\right]\psi(x)\frac{a^{\prime}\left[r\Delta_{n}+\frac{(c_{a}T_{n})^{\frac{1}{\kappa-1}}}{z}\right]}{a^{\prime}(T_{n}^{\frac{1}{\kappa-1}})}\right|\leq C\left(\left|s\right|^{2}\land\left|s\right|^{\theta}\right)s^{-\kappa-\varepsilon},
\]
which for $\varepsilon$ small enough, allow us to apply the Dominated
Convergence Theorem in order to get that 
\[
I_{n}^{\Delta,1}\left(B_{n}z\right)\rightarrow z^{\kappa-1}\int_{0}^{\infty}\psi(x)x^{-\kappa}\mathrm{d}x.
\]
The result now follows from Fubini's Theorem and the relation (see
Lemma 14.11 in \cite{Sato99})
\[
\int_{0}^{\infty}(e^{\pm\mathbf{i}r}\pm\mathbf{i}r-1)r^{-\varpi-1}\mathrm{d}\theta=\frac{\Gamma(2-\varpi)}{(\varpi-1)\varpi}e^{\mp\mathbf{i}\frac{\pi\varpi}{2}},\,\,\,\,1<\varpi<2.
\]

\textbf{Convergence of ii.:} \textbf{We conclude the proof by showing
that ii. }holds, so for the rest of the proof we will assume that
$\nu^{\pm}(x)\sim\tilde{K}_{\pm}x^{-\beta_{\nu}}$ as $x\rightarrow0^{+}$
and that $1>\beta_{\nu}>\kappa-1$. From the first and second part
of the proof, it is enough to show that in this situation
\begin{align}
I_{n}^{\Delta,1}\left(B_{n}z\right) & \rightarrow(\kappa-1)\int_{0}^{1}(1-s)s^{\beta_{\nu}-\kappa}\mathrm{d}s\psi_{\beta_{\nu}}(z);\label{limitchfs1}\\
I_{n}^{\Delta,2}\left(B_{n}z\right) & \rightarrow\int_{0}^{1}s^{\beta-\kappa+1}\mathrm{d}s\psi_{\beta_{\nu}}(z),\,\,l=2,3;\label{limitchfs2,3}\\
\mathcal{C}\left(z\ddagger B_{n}S_{n}^{\Delta_{n},4}\right) & \rightarrow\frac{1}{\kappa-2}\psi_{\beta_{\nu}}(z).\label{limitchfs4}
\end{align}
From Assumption \ref{assumptionlongmem-1} $a\in\mathrm{RV_{\alpha}^{\infty}},$
from which it follows easily that $r_{n}\rightarrow0$ as $n\rightarrow\infty$.
Thus, by KT and (\ref{chfctat0}), we deduce that as $n\rightarrow\infty$
\[
\mathcal{C}\left(z\ddagger B_{n}S_{n}^{\Delta_{n},4}\right)=r_{n}\psi\left(\frac{z}{r_{n}^{1/\beta_{\nu}}}\right)\frac{1}{r_{n}}\int_{T_{n}}^{\infty}a(s)\mathrm{d}s\rightarrow\frac{1}{\kappa-2}\psi_{\beta_{\nu}}(z),\,\,\,z\in\mathbb{R},
\]
On the other hand, by doing the change of variables $x=s/T_{n}$ we
get that 
\begin{align*}
I_{n}^{\Delta,1}\left(B_{n}z\right) & =\frac{T_{n}a^{\prime}(T_{n})}{a(T_{n})}\int_{0}^{1}\int_{0}^{1}(1-s)r_{n}\psi\left(\frac{zs}{r_{n}^{1/\beta_{\nu}}}\right)\frac{a^{\prime}(T_{n}(r/n+s))}{a^{\prime}(T_{n})}\mathrm{d}s\mathrm{d}r,\\
I_{n}^{\Delta,2}\left(B_{n}z\right)= & \int_{0}^{1}r_{n}\psi\left(\frac{zs}{r_{n}^{1/\beta_{\nu}}}\right)\frac{a(T_{n}s)}{a(T_{n})}\mathrm{d}s,
\end{align*}
From (\ref{chfctat0}) and the fact that $a^{\prime}\in\mathrm{RV_{\kappa}^{\infty}},$it
follows that for $0<r,s\leq1$ as $n\rightarrow\infty$
\begin{align*}
\frac{T_{n}a^{\prime}(T_{n})}{a(T_{n})}r_{n}\psi\left(\frac{zs}{r_{n}^{1/\beta_{\nu}}}\right)\frac{a^{\prime}(T_{n}(r/n+s))}{a^{\prime}(T_{n})} & \rightarrow(\kappa-1)\psi_{\beta_{\nu}}(zs)s^{-\kappa},\\
r_{n}\psi\left(\frac{zs}{r_{n}^{1/\beta_{\nu}}}\right)\frac{a(T_{n}s)}{a(T_{n})}\rightarrow & \psi_{\beta_{\nu}}(zs)s^{-\kappa+1}.
\end{align*}
Moreover, from (\ref{potterboundsfortrawlfunct}) and (\ref{boundchfncgethoorbig})
we infer that 
\begin{align*}
\left|r_{n}\psi\left(\frac{zs}{r_{n}^{1/\beta_{\nu}}}\right)\frac{a^{\prime}(T_{n}(r/n+s))}{a^{\prime}(T_{n})}\right| & \leq Cs^{\beta_{\nu}-\kappa-\varepsilon},\\
\left|r_{n}\psi\left(\frac{zs}{r_{n}^{1/\beta_{\nu}}}\right)\frac{a(T_{n}s)}{a(T_{n})}\right| & \leq Cs^{\beta_{\nu}+1-\kappa-\varepsilon}.
\end{align*}
Therefore, (\ref{limitchfs1}) and (\ref{limitchfs2,3}) now follows
by letting $\beta_{\nu}-\kappa+1>\varepsilon>0$ and an application
of the Dominated Convergence Theorem.

\end{proof}

\subsection{Proof of Theorem \ref{theoremmean-2-1}}

\begin{proof}[Proof of Theorem \ref{theoremmean-2-1}]We will only
show that \textbf{i.} holds, since the proof of \textbf{ii. }is identical
as the proof of ii. \textbf{in Theorem }\ref{theoremmean3}. Observe
that by the strict stability of $L$ and KT for all $z\in\mathbb{R}$
and $l=2,3,4$
\[
\mathcal{C}\left(z\ddagger\frac{\Delta_{n}}{T_{n}^{1/\beta}}S_{n}^{\Delta_{n},l}\right)\leq CT_{n}^{\beta}a(T_{n})\mathrm{O}(1)\rightarrow0,\,\,\,n\rightarrow\infty.
\]
This implies that $\frac{\Delta_{n}}{T_{n}^{1/\beta}}S_{\left[tn\right]}^{\Delta_{n}}=\frac{\Delta_{n}}{T_{n}^{1/\beta}}S_{\left[tn\right]}^{\Delta_{n},1}+\mathrm{o}_{\mathbb{P}}$.
Therefore, as in the proof of \textbf{Theorem }\ref{theoremmean3},
we only need to show that as $n\rightarrow\infty$
\[
\mathcal{C}\left(z\ddagger B_{n}S_{\left[nt\right]}^{\Delta_{n}}\right)\rightarrow\mathcal{C}(z\ddagger Y_{t}),\,\,\,t\geq0.
\]
To see this, first note that by Assumption \ref{assumptionlongmem-1}
$\int_{0}^{\infty}s^{\beta}a^{\prime}(s)\mathrm{d}s<\infty$ and 

\begin{align*}
\mathcal{C}\left(z\ddagger\frac{\Delta_{n}}{T_{n}^{1/\beta}}S_{\left[tn\right]}^{\Delta_{n},1}\right) & =\psi\left(z\right)\int_{0}^{1}\sum_{j=1}^{\left[nt\right]-1}\int_{t_{j}}^{t_{j+1}}\left(\frac{\left[nt\right]}{n}-\frac{s}{T_{n}}\right)t_{j}^{\beta}a^{\prime}(r\Delta_{n}+s)\mathrm{d}s\mathrm{d}r\\
 & +\psi\left(z\right)\frac{1}{T_{n}}\int_{0}^{1}\sum_{j=1}^{\left[nt\right]-1}\int_{t_{j}}^{t_{j+1}}\left(s-t_{j}\right)t_{j}^{\beta}a^{\prime}(r\Delta_{n}+s)\mathrm{d}s\mathrm{d}r.
\end{align*}
in which
\[
\left|\frac{1}{T_{n}}\int_{0}^{1}\sum_{j=1}^{\left[nt\right]-1}\int_{t_{j}}^{t_{j+1}}\left(s-t_{j}\right)t_{j}^{\beta}a^{\prime}(r\Delta_{n}+s)\mathrm{d}s\mathrm{d}r\right|\leq\frac{1}{n}\int_{0}^{\infty}s^{\beta}a^{\prime}(s)\mathrm{d}s\rightarrow0,
\]
as well as for $t_{j}\leq s\leq t_{j+1}$ and $0\leq r\leq1$
\[
\left|\left(\frac{\left[nt\right]}{n}-\frac{s}{T_{n}}\right)t_{j}^{\beta}a^{\prime}(r\Delta_{n}+s)\mathrm{d}s\mathrm{d}r\right|\leq s^{\beta}a^{\prime}(r\Delta_{n}+s).
\]
Hence, the Generalized Dominated Convergence Theorem asserts that
\[
\mathcal{C}\left(z\ddagger\frac{\Delta_{n}}{T_{n}^{1/\beta}}S_{\left[tn\right]}^{\Delta_{n},1}\right)\rightarrow t\psi\left(z\right)\int_{0}^{\infty}s^{\beta}a^{\prime}(s)\mathrm{d}s,
\]
as required.

\end{proof}
\subsection{Proof of Theorem \ref{TheSpecific}}
\begin{proof}[Proof of Theorem \ref{TheSpecific}]
	Consider the cumulant of the finite dimensional distribution of $\frac{1}{\sqrt{n}}\left(X_{t}^{(n)}-\mathbb{E}[X_{t}^{(n)}]\right)$,
	\begin{equation*}
	\log\mathbb{E}\left(\exp\left(i\sum_{j=1}^{k}z_{j}\frac{1}{\sqrt{n}}\left(X^{(n)}_{t_{j}}-\mathbb{E}[X^{(n)}_{t_{j}}]\right)\right)\right)=\int_{\mathbb{R}} \int_{\mathbb{R}} \bigg(-\frac{1}{2}\frac{(b^{(n)})^{2}}{n}\bigg(\sum_{j=1}^{k}z_{j}\textbf{I}_{A_{t_{j}}}(y,s)\bigg)^{2}
	\end{equation*}
	\begin{equation*}
	+\int_{\mathbb{R}}e^{i\frac{1}{\sqrt{n}}x\sum_{j=1}^{k}z_{j}\textbf{I}_{A_{t_{j}}}(y,s)}-1 -i\textbf{I}_{[-1,1]}(x)\frac{1}{\sqrt{n}}x \sum_{j=1}^{k}z_{j}\textbf{I}_{A_{t_{j}}}(y,s) \nu^{(n)}(dx)
	\end{equation*}
	\begin{equation*}
	-i\int_{|x|>1}x \frac{1}{\sqrt{n}}\sum_{j=1}^{k}z_{j}\textbf{I}_{A_{t_{j}}}(y,s)\nu^{(n)}(dx)\bigg)dsdy
	\end{equation*}
	\begin{equation*}
	=-\int_{\mathbb{R}} \int_{\mathbb{R}} \frac{1}{2}\frac{(b^{(n)})^{2}}{n}\bigg(\sum_{j=1}^{k}z_{j}\textbf{I}_{A_{t_{j}}}(y,s)\bigg)^{2}dsdy-\int_{\mathbb{R}}\int_{\mathbb{R}}\int_{\mathbb{R}}\frac{x^{2}}{n}\bigg(\sum_{j=1}^{k}z_{j}\textbf{I}_{A_{t_{j}}}(y,s)\bigg)^{2}\nu^{(n)}(dx)dsdy
	\end{equation*}
	\begin{equation*}
	+\int_{\mathbb{R}}\int_{\mathbb{R}}\int_{\mathbb{R}}e^{i\frac{1}{\sqrt{n}}  x \sum_{j=1}^{k}z_{j}\textbf{I}_{A_{t_{j}}}(y,s)}-1-i\frac{1}{\sqrt{n}} x\sum_{j=1}^{k}z_{j}\textbf{I}_{A_{t_{j}}}(y,s) 
	+\frac{1}{2}\frac{x^{2}}{n}\bigg(\sum_{j=1}^{k}z_{j}\textbf{I}_{A_{t_{j}}}(y,s)\bigg)^{2}\nu^{(n)}(dx)dsdy.
	\end{equation*}
	It is possible to observe that the last addendum is bounded by
	\begin{equation*}
	\int_{\mathbb{R}}\int_{\mathbb{R}}\int_{\mathbb{R}}\frac{\left(\sum_{j=1}^{k}z_{j}\textbf{I}_{A_{t_{j}}}(y,s)\right)^{3}}{6}\frac{|x|^{3}}{n\sqrt{n}}\nu^{(n)}(dx)dsdy=C \int_{\mathbb{R}}\frac{|x|^{3}}{n\sqrt{n}}\nu^{(n)}(dx),
	\end{equation*}
	where $C$ is a positive constant. Moreover, notice that
	\begin{equation*}
	\int_{\mathbb{R}}\int_{\mathbb{R}}\int_{\mathbb{R}}\frac{\left(\sum_{j=1}^{k}z_{j}\textbf{I}_{A_{t_{j}}}(y,s)\right)^{2}}{2}\frac{x^{2}}{n}\nu^{(n)}(dx)dsdy=\frac{1}{2}\sum_{j,l=1}^{k}z_{j}z_{l}Leb(A_{t_{j}}\cap A_{t_{l}})\int_{\mathbb{R}}\frac{x^{2}}{n}\nu^{(n)}(dx)
	\end{equation*}
	\begin{equation*}
	\rightarrow \frac{1}{2}\sum_{j,l=1}^{k}z_{j}z_{l}Leb(A_{t_{j}}\cap A_{t_{l}}), \quad\quad\textnormal{as $n\rightarrow\infty$.}
	\end{equation*}
	Thus, we have $Leb(A_{t_{j}}\cap A_{t_{l}})=Leb(A_{t_{j}-t_{l}}\cap A_{0})=\int_{0}^{\infty}a(t_{j}+t_{l}+s)ds=\int_{-\infty}^{\min(t_{j},t_{l})}g(t_{j}-s)g(t_{l}-s)ds$ for $j,l=1,...,k$.
	
	Therefore, in general terms we have the following condition to satisfy
	\begin{equation*}
	\int_{0}^{\infty}a(h+s)ds=\int_{-\infty}^{0}g(h-s)g(-s)ds.
	\end{equation*}
	Notice that
	\begin{equation*}
	\int_{0}^{\infty}a(h+s)ds=\int_{h}^{\infty}a(s)ds\quad\textnormal{and}\quad \int_{-\infty}^{0}g(h-s)g(-s)ds=\int_{0}^{\infty}g(h+s)g(s)ds,
	\end{equation*}
	by the fundamental theorem of calculus we have
	\begin{equation*}
	-a(h)=\frac{d}{dh}\int_{h}^{\infty}a(s)ds=\frac{d}{dh}\int_{0}^{\infty}g(h+s)g(s)ds.
	\end{equation*}
	 In this step we prove that the sequence of probability laws generated by our process is tight. Recall that from Theorem 13.5 in \cite{Billingsley99} in order to prove tightness is sufficient to prove that for any $t\geq s\geq r\in\mathbb{R}$, $n\geq 1$
	 \begin{equation*}
	 \frac{1}{n^{\beta}}\mathbb{E}\left[\left|X_{t}^{(n)}-\mathbb{E}[X_{t}^{(n)}]-X_{s}^{(n)}+\mathbb{E}[X_{s}^{(n)}]\right|^{\beta}\left|X_{s}^{(n)}-\mathbb{E}[X_{s}^{(n)}]-X_{r}^{(n)}+\mathbb{E}[X_{r}^{(n)}]\right|^{\beta}\right]\leq (F(t)-F(r))^{\alpha},
	 \end{equation*}
	 where $F$ is a non-decreasing continuous function, $\beta>1$ and $\alpha>1$. We will see now how we split the trawl processes into independent processes. First, define (see Fig.~1)
	 \begin{figure}\label{Fig1}
	 	\centerline{\includegraphics[scale=1]{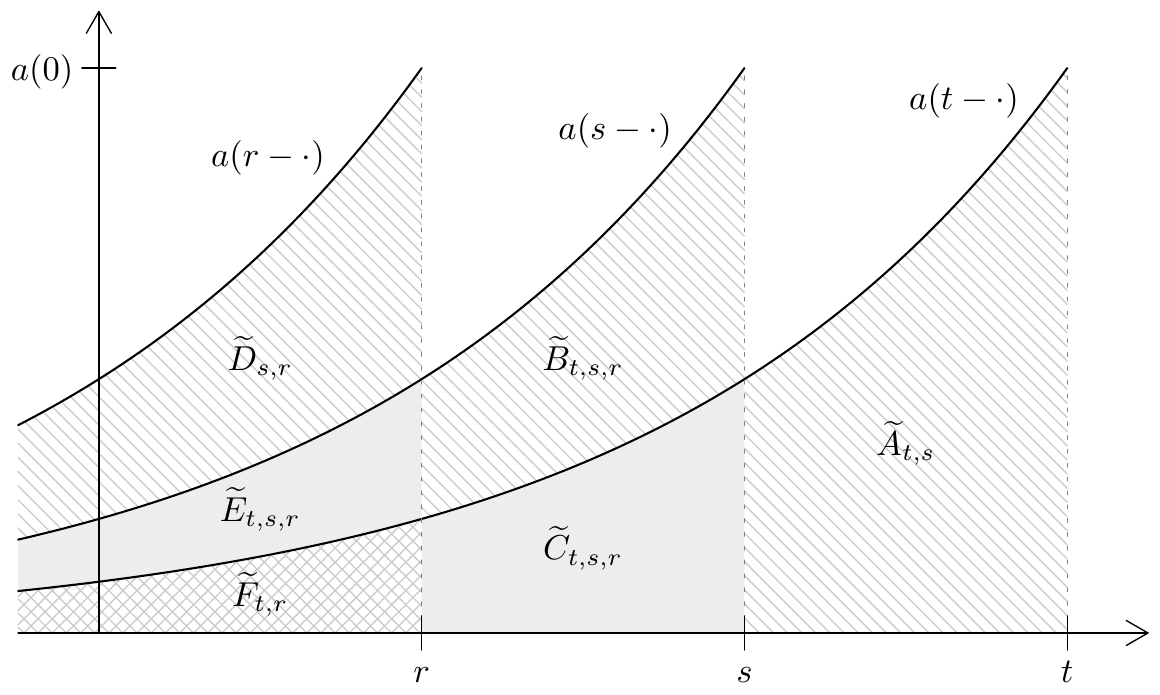}}
	 	\caption{\small Separation of \mikko{overlapping} trawl sets into disjoint sets.}
	 \end{figure}
	 \begin{equation*}
	 \tilde{A}_{t,s}:=A_{t}\setminus A_{s} ,\quad\tilde{B}_{t,s,r}:=A_{s}\setminus A_{t}\setminus A_{r},\quad \tilde{C}_{t,s,r}:=A_{s}\setminus A_{r}\setminus \tilde{B}_{t,s,r},
	 \end{equation*}
	 \begin{equation*}
	 \tilde{D}_{s,r}:=A_{r}\setminus A_{s} , \quad \tilde{E}_{t,s,r}:=A_{s}\setminus A_{t}\setminus \tilde{B}_{t,s,r}\quad\tilde{F}_{t,r}:= A_{t}\cap A_{r}.
	 \end{equation*}
	 Then we have the following almost sure equalities
	 \begin{equation*}
	 X_{t}^{(n)}-X_{s}^{(n)}=L^{(n)}(\tilde{A}_{t,s})-L^{(n)}(\tilde{B}_{t,s,r})-L^{(n)}(\tilde{E}_{t,s,r}),
	 \end{equation*}
	 \begin{equation*}
	 \textnormal{and}\quad X_{s}^{(n)}-X_{r}^{(n)}=L^{(n)}(\tilde{B}_{t,s,r})+ L^{(n)}(\tilde{C}_{t,s,r})-L^{(n)}(\tilde{D}_{s,r}).
	 \end{equation*}
	 Hence, regarding tightness, using the fact that $L$ is an independently scattered random measure, we have
	 \begin{equation*}
	 \frac{1}{n^{2}}\mathbb{E}\left[(X_{t}^{(n)}-\mathbb{E}[X_{t}^{(n)}]-X_{s}^{(n)}-\mathbb{E}[X_{s}^{(n)}])^{2}(X_{s}^{(n)}-\mathbb{E}[X_{s}^{(n)}]-X_{r}^{(n)}-\mathbb{E}[X_{r}^{(n)}])^{2}\right]
	 \end{equation*}
	 \begin{equation*}
	 =\frac{9}{n^{2}}\bigg\{Var\left(L^{(n)}(\tilde{A}_{t,s})\right)Var\left(L^{(n)}(\tilde{C}_{t,s,r})\right)+Var\left(L^{(n)}(\tilde{A}_{t,s})\right)Var\left(L^{(n)}(\tilde{D}_{s,r})\right)
	 \end{equation*}
	 \begin{equation*}
	 +Var\left(L^{(n)}(\tilde{A}_{t,s})\right)Var\left(L^{(n)}(\tilde{B}_{t,s,r})\right)+Var\left(L^{(n)}(\tilde{E}_{t,s,r})\right)Var\left(L^{(n)}(\tilde{C}_{t,s,r})\right)
	 \end{equation*}
	 \begin{equation*}
	 +Var\left(L^{(n)}(\tilde{E}_{t,s,r})^{2}\right)Var\left(L^{(n)}(\tilde{D}_{s,r})\right)
	 +Var\left(L^{(n)}(\tilde{E}_{t,s,r})\right)Var\left(L^{(n)}(\tilde{B}_{t,s,r})\right)
	 \end{equation*}
	 \begin{equation*}
	 +Var\left(L^{(n)}(\tilde{B}_{t,s,r})\right)Var\left(L^{(n)}(\tilde{C}_{t,s,r})\right)+Var\left(L^{(n)}(\tilde{B}_{t,s,r})\right)Var\left(L^{(n)}(\tilde{D}_{s,r})\right)
	 \end{equation*}
	 \begin{equation}\label{long}
	 +\mathbb{E}\left[\Big(L^{(n)}(\tilde{B}_{t,s,r})-\mathbb{E}[L^{(n)}(\tilde{B}_{t,s,r})]\Big)^{4}\right]\bigg\}.
	 \end{equation}
	 Let us concentrate first on
	 \begin{equation*}
	 Var\left(L^{(n)}(\tilde{A}_{t,s})\right)Var\left(L^{(n)}(\tilde{D}_{s,r})\right).
	 \end{equation*}
	 We have
	 \begin{equation*}
	 Var\left(L^{(n)}(\tilde{A}_{t,s})\right)Var\left(L^{(n)}(\tilde{D}_{s,r})\right)=\frac{9}{n^{2}}\left((b^{(n)})^{2}+\int_{\mathbb{R}}x^{2}\nu^{(n)}(dx)\right)^{2} Leb(\tilde{A}_{t,s})Leb(\tilde{D}_{s,r})
	 \end{equation*}
	 \begin{equation*}
	 \leq C \int_{0}^{t-s}a(p)dp \int_{0}^{s-r}a(p)dp
	 \end{equation*}
	 	where we used that $\frac{(b^{(n)})^{2}}{n}\rightarrow0$, that $\int_{\mathbb{R}}\frac{x^{2}}{n}\nu^{(n)}(dx)\rightarrow 1$ as $n\rightarrow\infty$, and that, due to the homogeneity of the L\'{e}vy basis and the monotonicity of $a$, we have
	 \begin{equation*}
	 Leb(\tilde{D}_{s,r})=Leb(A_{r}\setminus A_{s})=Leb(A_{s}\setminus A_{r})= \int_{0}^{s-r}a(p)dp.
	 \end{equation*}
	 Now by Assumption \ref{A1} we have
	 \begin{equation*}
	 C \int_{0}^{t-s}a(p)dp \int_{0}^{s-r}a(p)dp\leq C' (t-s)^{\frac{1}{2}+\frac{\epsilon}{2}} (s-r)^{\frac{1}{2}+\frac{\epsilon}{2}}=C''(t-r)^{1+\epsilon}.
	 \end{equation*}
	 The same arguments applies to all the other summands in $(\ref{long})$ except 
	 \begin{equation*}
	 \frac{9}{n^{2}}\mathbb{E}\left[\Big(L^{(n)}(\tilde{B}_{t,s,r})-\mathbb{E}[L^{(n)}(\tilde{B}_{t,s,r})]\Big)^{4}\right].
	 \end{equation*}
	 From the Appendix, we have that
	 \begin{equation*}
	 \mathbb{E}\left[\Big(L^{(n)}(\tilde{B}_{t,s,r})-\mathbb{E}[L^{(n)}(\tilde{B}_{t,s,r})]\Big)^{4}\right]=Leb(\tilde{B}_{t,s,r})^{2} \left((b^{(n)})^{2}+\int_{\mathbb{R}}x^{2}\nu^{(n)}(dx)\right)^{2}+Leb(\tilde{B}_{t,s,r}) \left(\int_{\mathbb{R}} x^{4}\nu^{(n)}(dx)\right).
	 \end{equation*}
	 Moreover, by the Assumption \ref{A1} on $C_{T}$ we have $Leb(\tilde{B}_{t,s,r})\geq \left(Leb(\tilde{B}_{t,s,r})\right)^{\beta}$ for any integer $\beta\geq1$, and by assumption of the theorem we have
	 \begin{equation*}
	 \frac{9}{n^{2}}\left[\left((b^{(n)})^{2}+\int_{\mathbb{R}}x^{2}\nu^{(n)}(dx)\right)^{2}+\int_{\mathbb{R}}x^{4}\nu^{(n)}(dx)\right]< K,
	 \end{equation*}
	 for some positive constant $K$. Thus, we obtain that
	 \begin{equation*}
	 \mathbb{E}\left[\Big(L^{(n)}(\tilde{B}_{t,s,r})-\mathbb{E}[L^{(n)}(\tilde{B}_{t,s,r})]\Big)^{4}\right]\leq KLeb(\tilde{B}_{t,s,r}).
	 \end{equation*}
	 Then, by Assumption \ref{A1} we deduce that
	 \begin{equation*}
	 \mathbb{E}\left[\Big(L^{(n)}(\tilde{B}_{t,s,r})-\mathbb{E}[L^{(n)}(\tilde{B}_{t,s,r})]\Big)^{4}\right]
	 \leq K'(t-r)^{1+\epsilon},
	 \end{equation*}
	 for some positive constant $K'$. Therefore, we conclude that
	 \begin{equation*}
	 	\frac{1}{n^{2}}\mathbb{E}\left[(X_{t}^{(n)}-\mathbb{E}[X_{t}^{(n)}]-X_{s}^{(n)}-\mathbb{E}[X_{s}^{(n)}])^{2}(X_{s}^{(n)}-\mathbb{E}[X_{s}^{(n)}]-X_{r}^{(n)}-\mathbb{E}[X_{r}^{(n)}])^{2}\right]\leq C(t-r)^{1+\epsilon}.
	 \end{equation*}
\end{proof}
\section*{Appendix: The fourth moment of the trawl process}
\begin{equation*}
\mathbb{E}(X_{t}^{4})=\frac{\partial^{4}}{\partial z^{4}}\exp\left(\int_{\mathbb{R}}\int_{\mathbb{R}}\log\mathbb{E}\left(e^{iz \textbf{I}_{A_{t}}(s) L'(x)}\right)dsdx\right)\bigg|_{z=0}
\end{equation*}
\begin{equation*}
=\frac{\partial^{4}}{\partial z^{4}}\exp\left(\int_{\mathbb{R}}\int_{\mathbb{R}}\left(iz\textbf{I}_{A_{t}}(x,s)\gamma-\frac{1}{2}b^{2}\textbf{I}_{A_{t}}(x,s)z^{2}+\int_{\mathbb{R}}e^{iz\textbf{I}_{A_{t}}(x,s) y}-1 -i\textbf{I}_{[-1,1]}(y)\textbf{I}_{A_{t}}(x,s)z y\nu(dy)\right)dsdx\right)\bigg|_{z=0}
\end{equation*}
\begin{equation*}
=\frac{\partial^{3}}{\partial z^{3}}\exp\left(\int_{\mathbb{R}}\int_{\mathbb{R}}\left(iz\textbf{I}_{A_{t}}(x,s)\gamma-\frac{1}{2}b^{2}\textbf{I}_{A_{t}}(x,s)z^{2}+\int_{\mathbb{R}}e^{iz\textbf{I}_{A_{t}}(x,s) y}-1 -i\textbf{I}_{[-1,1]}(y)\textbf{I}_{A_{t}}(x,s)z y\nu(dy)\right)dxds\right)
\end{equation*}
\begin{equation*}
\int_{\mathbb{R}}\int_{\mathbb{R}}\left(i\textbf{I}_{A_{t}}(x,s)\gamma-b^{2}\textbf{I}_{A_{t}}(x,s)z+\int_{\mathbb{R}}i \textbf{I}_{A_{t}}(x,s)y e^{iz \textbf{I}_{A_{t}}(x,s) y}-i\textbf{I}_{[-1,1]}(y)\textbf{I}_{A_{t}}(x,s)y\nu(dy)\right)dxds\bigg|_{z=0}
\end{equation*}
\begin{equation*}
=\frac{\partial^{2}}{\partial z^{2}}\exp\left(\int_{\mathbb{R}}\int_{\mathbb{R}}\left(iz\textbf{I}_{A_{t}}(x,s)\gamma-\frac{1}{2}b^{2}\textbf{I}_{A_{t}}(x,s)z^{2}+\int_{\mathbb{R}}e^{iz\textbf{I}_{A_{t}}(x,s) y}-1 -i\textbf{I}_{[-1,1]}(y)\textbf{I}_{A_{t}}(x,s)z y\nu(dy)\right)dxds\right)
\end{equation*}
\begin{equation*}
\left(\int_{\mathbb{R}}\int_{\mathbb{R}}\left(i\textbf{I}_{A_{t}}(x,s)\gamma-b^{2}\textbf{I}_{A_{t}}(x,s)z+\int_{\mathbb{R}}i \textbf{I}_{A_{t}}(x,s)y e^{iz \textbf{I}_{A_{t}}(x,s) y}-i\textbf{I}_{[-1,1]}(y)\textbf{I}_{A_{t}}(x,s)y\nu(dy)\right)dxds\right)^{2}\bigg|_{z=0}
\end{equation*}
\begin{equation*}
+ \frac{\partial^{2}}{\partial z^{2}}\exp\left(\int_{\mathbb{R}}\int_{\mathbb{R}}\left(iz\textbf{I}_{A_{t}}(x,s)\gamma-\frac{1}{2}b^{2}\textbf{I}_{A_{t}}(x,s)z^{2}+\int_{\mathbb{R}}e^{iz\textbf{I}_{A_{t}}(x,s) y}-1 -i\textbf{I}_{[-1,1]}(y)\textbf{I}_{A_{t}}(x,s)z y\nu(dy)\right)dxds\right)
\end{equation*}
\begin{equation*}
\int_{\mathbb{R}}\int_{\mathbb{R}}\left(-b^{2}\textbf{I}_{A_{t}}(x,s)-\int_{\mathbb{R}}\textbf{I}_{A_{t}}(x,s)y^{2} e^{iz \textbf{I}_{A_{t}}(x,s) y}\nu(dy)\right)dxds\bigg|_{z=0}
\end{equation*}
\begin{equation*}
=\frac{\partial}{\partial z}\exp\left(\int_{\mathbb{R}}\int_{\mathbb{R}}\left(iz\textbf{I}_{A_{t}}(x,s)\gamma-\frac{1}{2}b^{2}\textbf{I}_{A_{t}}(x,s)z^{2}+\int_{\mathbb{R}}e^{iz\textbf{I}_{A_{t}}(x,s) y}-1 -i\textbf{I}_{[-1,1]}(y)\textbf{I}_{A_{t}}(x,s)z y\nu(dy)\right)dxds\right)
\end{equation*}
\begin{equation*}
\left(\int_{\mathbb{R}}\int_{\mathbb{R}}\left(i\textbf{I}_{A_{t}}(x,s)\gamma-b^{2}\textbf{I}_{A_{t}}(x,s)z+\int_{\mathbb{R}}i \textbf{I}_{A_{t}}(x,s)y e^{iz \textbf{I}_{A_{t}}(x,s) y}-i\textbf{I}_{[-1,1]}(y)\textbf{I}_{A_{t}}(x,s)y\nu(dy)\right)dxds\right)^{3}\bigg|_{z=0}
\end{equation*}
\begin{equation*}
+ \frac{\partial}{\partial z}\exp\left(\int_{\mathbb{R}}\int_{\mathbb{R}}\left(iz\textbf{I}_{A_{t}}(x,s)\gamma-\frac{1}{2}b^{2}\textbf{I}_{A_{t}}(x,s)z^{2}+\int_{\mathbb{R}}e^{iz\textbf{I}_{A_{t}}(x,s) y}-1 -i\textbf{I}_{[-1,1]}(y)\textbf{I}_{A_{t}}(x,s)z y\nu(dy)\right)dxds\right)
\end{equation*}
\begin{equation*}
3\left(\int_{\mathbb{R}}\int_{\mathbb{R}}\left(i\textbf{I}_{A_{t}}(x,s)\gamma-b^{2}\textbf{I}_{A_{t}}(x,s)z+\int_{\mathbb{R}}i \textbf{I}_{A_{t}}(x,s)y e^{iz \textbf{I}_{A_{t}}(x,s) y}-i\textbf{I}_{[-1,1]}(y)\textbf{I}_{A_{t}}(x,s)y\nu(dy)\right)dxds\right)
\end{equation*}
\begin{equation*}
\int_{\mathbb{R}}\int_{\mathbb{R}}\left(-b^{2}\textbf{I}_{A_{t}}(x,s)-\int_{\mathbb{R}}\textbf{I}_{A_{t}}(x,s)y^{2} e^{iz \textbf{I}_{A_{t}}(x,s) y}\nu(dy)\right)dxds\bigg|_{z=0}
\end{equation*}
\begin{equation*}
+ \frac{\partial}{\partial z}\exp\left(\int_{\mathbb{R}}\int_{\mathbb{R}}\left(iz\textbf{I}_{A_{t}}(x,s)\gamma-\frac{1}{2}b^{2}\textbf{I}_{A_{t}}(x,s)z^{2}+\int_{\mathbb{R}}e^{iz\textbf{I}_{A_{t}}(x,s) y}-1 -i\textbf{I}_{[-1,1]}(y)\textbf{I}_{A_{t}}(x,s)z y\nu(dy)\right)dxds\right)
\end{equation*}
\begin{equation*}
\int_{\mathbb{R}}\int_{\mathbb{R}}\left(-i\int_{\mathbb{R}}\textbf{I}_{A_{t}}(x,s)y^{3} e^{iz \textbf{I}_{A_{t}}(x,s) y}\nu(dy)\right)dxds\bigg|_{z=0}
\end{equation*}
\begin{equation*}
=\left(\int_{\mathbb{R}}\int_{\mathbb{R}}\left(\textbf{I}_{A_{t}}(x,s)\gamma+\int_{|y|>1}\textbf{I}_{A_{t}}(x,s)y\nu(dy)\right)dxds\right)^{4}
\end{equation*}
\begin{equation*}
+6 \left(\int_{\mathbb{R}}\int_{\mathbb{R}}\left(\textbf{I}_{A_{t}}(x,s)\gamma+\int_{|y|>1}\textbf{I}_{A_{t}}(x,s)y\nu(dy)\right)dxds\right)^{2} \int_{\mathbb{R}}\int_{\mathbb{R}}\left(b^{2}\textbf{I}_{A_{t}}(x,s)+\int_{\mathbb{R}}\textbf{I}_{A_{t}}(x,s)y^{2}\nu(dy)\right)dxds
\end{equation*}
\begin{equation*}
+3 \left(\int_{\mathbb{R}}\int_{\mathbb{R}}\left(b^{2}\textbf{I}_{A_{t}}(x,s)+\int_{\mathbb{R}}\textbf{I}_{A_{t}}(x,s)y^{2}\nu(dy)\right)dxds\right)^{2}
\end{equation*}
\begin{equation*}
+4\left(\int_{\mathbb{R}}\int_{\mathbb{R}}\left(\textbf{I}_{A_{t}}(x,s)\gamma+\int_{|y|>1}\textbf{I}_{A_{t}}(x,s)y\nu(dy)\right)dxds\right)\int_{\mathbb{R}}\int_{\mathbb{R}}\left(\int_{\mathbb{R}}\textbf{I}_{A_{t}}(x,s)y^{3} \nu(dy)\right)dxds
\end{equation*}
\begin{equation*}
+ \int_{\mathbb{R}}\int_{\mathbb{R}}\left(\int_{\mathbb{R}}\textbf{I}_{A_{t}}(x,s)y^{4}\nu(dy)\right)dxds
\end{equation*}
\begin{equation*}
=Leb(A_{t})^{4} \left(\gamma+\int_{|y|>1} y\nu(dy)\right)^{4}+Leb(A_{t})^{3}\left(\gamma+\int_{|y|>1} y\nu(dy)\right)^{2}\left(b^{2}+\int_{\mathbb{R}} y^{2}\nu(dy)\right)
\end{equation*}
\begin{equation*}
+Leb(A_{t})^{2} \left[\left(b^{2}+\int_{\mathbb{R}} y^{2}\nu(dy)\right)^{2}+\left(\int_{\mathbb{R}} y^{3}\nu(dy)\right)\left(\gamma+\int_{|y|>1} y\nu(dy)\right)\right]+Leb(A_{t}) \left(\int_{\mathbb{R}} y^{4}\nu(dy)\right)
\end{equation*}
Moreover, by similar computations we have that
\begin{equation*}
\mathbb{E}\left[(X_{t}-\mathbb{E}[X_{t}])^{4}\right]=Leb(A_{t})^{2} \left(b^{2}+\int_{\mathbb{R}} y^{2}\nu(dy)\right)^{2}+Leb(A_{t}) \left(\int_{\mathbb{R}} y^{4}\nu(dy)\right).
\end{equation*}

\end{document}